\documentclass[12pt,a4paper, oneside,reqno]{article}

\usepackage[utf8]{inputenc}
\usepackage[english]{babel}

\usepackage[left=1in,right=1in,top=1in,bottom=1in]{geometry}

\usepackage{amsfonts,amssymb,amsmath,amsthm}

\usepackage[symbol]{footmisc} 
\usepackage{longtable}
\usepackage{tasks}

\usepackage{enumerate}
\usepackage{cite}
\usepackage{hyperref}

\usepackage{verbatim}
\usepackage{tkz-graph}

\usepackage{fancyhdr} 
\theoremstyle{plain}
\newtheorem{theorem}{Theorem}
\newtheorem{lemma}[theorem]{Lemma}
\newtheorem{proposition}[theorem]{Proposition}

\newtheorem{definition}[theorem]{Definition}
\newtheorem{corollary}[theorem]{Corollary}

\usetikzlibrary{arrows}

\usepackage{fouriernc} 

\begin{document}

\noindent{\Large 
Quasi-derivations of   Witt and related  algebras}
 \footnote{
 The  authors would like to thank the SRMC (Sino-Russian Mathematics Center in Peking University, Beijing, China) for its hospitality and excellent working conditions, where some parts of this work have been done.
The  work is supported by 
FCT   2023.08031.CEECIND and UID/00212/2023.}

	\bigskip
	
	 \bigskip
	
	 \bigskip

\begin{center}	
	{\bf
		Ivan Kaygorodov\footnote{CMA-UBI, University of  Beira Interior, Covilh\~{a}, Portugal; 
  \    kaygorodov.ivan@gmail.com},
   Abror Khudoyberdiyev\footnote{V.I.Romanovskiy Institute of Mathematics Academy of Science of Uzbekistan; National University of Uzbekistan; \ 
khabror@mail.ru} \&
Zarina Shermatova\footnote{V.I.Romanovskiy Institute of Mathematics Academy of Science of Uzbekistan; Kimyo International University in Tashkent;  \ 
zarina\_shermatova91@mail.ru}}   
\end{center}

 \bigskip

   \bigskip

\noindent {\bf Abstract.}
{\it 
In the present work, we compute   quasi-derivations of the Witt algebra and some algebras well-related to the Witt algebra. Namely, we  prove that each quasi-derivation of the Witt algebra is a sum of a derivation and a $\frac{1}{2}$-derivation; a similar result is obtained for the Virasoro algebra. 
A different situation appears for   Lie algebras ${\mathcal W}(a,b):$ in the case of $b=-1,$ they do not have interesting examples of quasi-derivations, but the case of $b\neq-1$ provides some new non-trivial examples of quasi-derivations. 
We also completely describe all quasi-derivations of  ${\mathcal W}(a,b).$
As a corollary, we describe the derivations and quasi-derivations of the Novikov-Witt and admissible Novikov-Witt algebras previously constructed by Bai and his co-authors;
and $\delta$-derivations and transposed $\delta$-Poisson structures on cited Lie algebras.
 In particular, we proved that each  ${\mathcal W}(a,b)$ admits a nontrivial transposed $\frac 1{1-b}$-Poisson structure.

}
 \bigskip

 \bigskip
\noindent {\bf Keywords}: 
{\it   Lie algebra, Witt algebra, Virasoro algebra,  generalized derivation, transposed $\delta$-Poisson structure.}

 \bigskip
\noindent {\bf MSC2020}: 17B68, 17A30, 17B40.

  \bigskip

  \bigskip

\tableofcontents

\newpage
\section*{Introduction} 
A notion of reverse derivations, as a particular case of Jordan derivations, was given in a paper by Herstein in 1957 \cite{her}.
In the anticommutative case, the notion of reverse derivations gives antiderivations,
which were actively studied by Filippov \cite{f95}. For example, he proved that
any prime Lie algebra with a nonzero antiderivation satisfies the standard identity of degree $5.$
Later, Filippov generalized the notion of derivations and antiderivations to $\delta$-derivations \cite{f98},
which were actively studied for algebras from various varieties (see \cite{k23,zz,z,k12} and references therein).
A notion, which generalized $\delta$-derivations, was introduced by Leger and Luks under the name generalized derivations, in 2000 \cite{LL}.
So, generalized derivations of triangular, alternative, and Jordan algebras were studied
by Jiménez-Gestal and Pérez-Izquierdo \cite{jp03}; 
Shestakov \cite{sh12}; 
and Martín Barquero, Martín González, Sánchez-Ortega, and Vandeyar \cite{MMS}.
On the other side, generalized derivations of Lie algebras give a particular case of near-derivations defined by Bre$\check{\textrm{s}}$ar in \cite{Br}, and generalized derivations are well-related to Jordan $\{g,h\}$-derivations introduced by Bre$\check{\textrm{s}}$ar in \cite{Br2}.
Generalized derivations for $n$-ary algebras were studied in \cite{ KP}.
Generalized derivations have appeared in various contexts.
So, they were used by   Andruszkiewicz, Brzezi\'nski, and Radziszewski to construct Lie affgebras in \cite{ABR}
and by Jiménez-Gestal and Pérez-Izquierdo in the study of finite-dimensional real division algebras \cite{jp}.
Burde and Dekimpe used generalized derivations of Lie algebras for studying post-Lie algebra structures \cite{BD}.
Each non-trivial $\delta$-derivation of a Lie algebra gives a ${\rm Hom}$-Lie structure \cite{f98}.
$\frac 12$-derivations play an important role in the definition of transposed Poisson algebras \cite{aae23,kk23,kkg23,kms,kkz1}.
 
\medskip 
The Witt algebra and its central extension (Virasoro algebra) are among the most popular infinite-dimensional Lie algebras.
The Witt algebra plays a role as a foundation for the construction of other important algebraic structures.
So, transposed Poisson structures on the Witt algebra were described in \cite{FKL};
admissible Novikov algebras on the Witt algebra were constructed in \cite{BG};
Witt Lie conformal algebras were studied in  \cite{akl}.
Also, the Witt algebra is a subalgebra in various infinite-dimensional Lie algebras (see \cite{kk23,kkg23,kkz1,tang2,gao16,gao,FKL,BG} and references therein).
The Witt algebra is under a certain pure algebraic interest now:
$\frac{1}{2}$-derivations were described in \cite{FKL};
 (anti)-Rota-Baxter operators were studied in \cite{gao16,Az};
derivations, extensions, and rigidity of subalgebras of the Witt algebra were studied by Buzaglo in \cite{B24};
Sierra and Walton proved that the universal enveloping algebra of the Witt algebra is not noetherian \cite{sw};
Sierra and Petukhov described the Poisson spectrum of the symmetric algebra of the Witt algebra \cite{sp};
Kong, Chen, Bai, and Tang studied left-symmetric structures on the Witt algebra \cite{TB,KCB}.
Let us also mention a big interest in the study of local mappings on generalizations of Witt algebras \cite{CYK,YK}
and a greater interest in the study of   modules and representations over generalized Witt algebras \cite{BY,GZ,XW,MZ,SST}.
A generalization of the Witt algebra, known as ${\mathcal {W}}(a,b)$ algebra, appeared in "Bombay Lectures" of Kac and Raina \cite{kac} (see also, \cite{mat}). 
Various linear mappings on ${\mathcal {W}}(a,b)$ are under an active investigation now \cite{FKL,gao,han,XDM,tang1,tang2}.
 
\medskip 

The paper is organized as follows.
Section \ref{pre} is dedicated to the main definitions used in the paper.
Section \ref{witt} describes quasi-derivations of the Witt algebra. As corollaries, we have the description of derivations and quasi-derivations of two types of Witt-admissible algebras, previously constructed in  \cite{TB, KCB, BG}.
Section \ref{vira} describes quasi-derivations of the Virasoro algebra.
Section \ref{wab} describes quasi-derivations and transposed $\delta$-Poisson structures of algebras ${\mathcal W}(a,b)$.

\newpage
\section{Preliminaries}\label{pre}

In this section, we recall some definitions and known results for studying generalized derivations and quasi-derivations of Lie algebras. 
All algebras and vector spaces are considered over the complex field. 
Let ${\rm S}$ be a set of vectors; we denote the linear span of ${\rm S}$ by $\langle {\rm S} \rangle.$

\begin{definition}
 Let ${\mathfrak L}$ be a Lie algebra.  The linear map ${\mathfrak D}:{\mathfrak L}\rightarrow {\mathfrak L}$
is called a \textit{derivation} if $${\mathfrak D}[x,y]\ =\ [{\mathfrak D}(x),y]+[x,{\mathfrak D}(y)].$$ 
As usual, we denote the set of all derivations of ${\mathfrak L}$ by $\mathfrak{Der}({\mathfrak L}).$ \end{definition}
Obviously,
$\mathfrak{Der}({\mathfrak L})$ is a Lie subalgebra of the general linear algebra $\mathfrak{gl}({\mathfrak L}).$ There are several generalizations of the notion of a derivation. In this paper, we consider generalized derivations and quasi-derivations of Lie
algebras as defined by Leger and Luks in \cite{LL}. 
\begin{definition}
    Let $f: {\mathfrak L} \rightarrow {\mathfrak L}$ is a linear map. If there exist linear maps $f',$ $ f'' : {\mathfrak L} \rightarrow {\mathfrak L}$ such that 
$$[f(x),y]+[x,f'(y)]\ =\ f''[x,y],$$
then $f$ is called a \textit{generalized derivation}. 
In case $f=f',$ then $f$ is said to be a \textit{quasi-derivation}.
In case $ f= f'= \delta f'',$ then $f''$ is said to be a \textit{$\delta$-derivation}.
\end{definition}

By ${\rm G}\mathfrak{Der}({\mathfrak L})$ we will denote the set of all generalized derivations of ${\mathfrak L};$  
by ${\rm Q}\mathfrak{Der}({\mathfrak L})$ the set of all quasi-derivations of ${\mathfrak L};$
by $\mathfrak{Der}_{\delta }({\mathfrak L})$ we will denote the set of all $\delta$-derivations (for a fixed $\delta$) of ${\mathfrak L};$
by $\mathfrak{Der}_{[\delta]}({\mathfrak L})$ we will denote the set of all $\delta$-derivations (for all $\delta \in \mathbb{C}$) of ${\mathfrak L}.$

Obviously, ${\rm Q}\mathfrak{Der}(\mathfrak L)$ and ${\rm G}\mathfrak{Der}(\mathfrak L)$
are Lie subalgebras of $\mathfrak{gl}(\mathfrak L)$ such that 
$$\mathfrak{Der}(\mathfrak L)\subseteq
\mathfrak{Der}_{[\delta]}({\mathfrak L}) \subseteq {\rm Q}\mathfrak{Der}(\mathfrak L)\subseteq {\rm G}\mathfrak{Der}(\mathfrak L)\subseteq \mathfrak{gl}(\mathfrak L).$$  
Another Lie subalgebra of  $\mathfrak{gl}(\mathfrak L)$ is the centroid  of $\mathfrak L,$ which is defined as  
$${\rm C}(\mathfrak L)= \big\{\varphi\in \mathfrak{gl}(\mathfrak L)\ |\  \varphi([x,y])=[x,\varphi(y)]  \big\}.$$
Thus ${\rm C}( \mathfrak L)\subseteq {\rm Q}\mathfrak{Der}(\mathfrak L)$ and so 
$\mathfrak{Der}(\mathfrak L)+{\rm C}(\mathfrak L)\subseteq {\rm Q}\mathfrak{Der}(\mathfrak L).$
In several cases, this is a strict inclusion. However, for some Lie algebras (see  \cite{LL,Br}) we have 
\begin{center} 
$ \mathfrak{Der}(\mathfrak L)+{\rm C}(\mathfrak L)= {\rm Q}\mathfrak{Der}(\mathfrak L)$ or  
$\mathfrak{Der}(\mathfrak L)+{\rm C}(\mathfrak L)= {\rm G}\mathfrak{Der}(\mathfrak L).$   
\end{center}
Note that for any centerless Lie algebra $\mathfrak L$, the sum 
$$\mathfrak{Der}(\mathfrak L)+{\rm C}(\mathfrak L)= \mathfrak{Der}(\mathfrak L)\oplus {\rm C}(\mathfrak{L})$$ is a direct sum of vector spaces. 

Recall that the notion of a quasi-centroid ${\rm QC}(\mathfrak L)$ of a Lie algebra $\mathfrak L$ was defined in \cite{LL} as
$${\rm QC}(\mathfrak L)\ =\ \big\{f\in \mathfrak{gl}(\mathfrak L)\ |\  [f(x),y]=[x,f(y)]\big\}.$$
Obviously, ${\rm C}(\mathfrak L)\subseteq {\rm QC}(\mathfrak L)$ and 
\begin{equation}\label{3}
 {\rm G}\mathfrak{Der}(\mathfrak L)={\rm Q}\mathfrak{Der}(\mathfrak L)+{\rm QC}(\mathfrak L)   
\end{equation}
(see \cite{LL}). Note that each commuting linear map $f: \mathfrak L\rightarrow \mathfrak  L$ $\big($that is, $[f(x),x]=0$ $\big)$ belongs to ${\rm QC}(\mathfrak L).$ Moreover, ${\rm QC}(\mathfrak L)$ coincides with the set of all commuting linear maps of $\mathfrak L.$
 Let $\mathfrak L$ be a centerless Lie algebra. 
  Suppose that $\mathfrak L$ is a Lie algebra such that ${\rm Z}_{\mathfrak L} \big([\mathfrak L,\mathfrak L]\big)=0.$ Then, by the result of Bre$\check{\textrm{s}}$ar and Zhao \cite{BZ} implies that the set of all commuting linear maps of $\mathfrak L$
coincides with ${\rm C}(\mathfrak L).$ Thus ${\rm QC}(\mathfrak L)={\rm C}(\mathfrak L)\subseteq {\rm Q}\mathfrak{Der}(\mathfrak L)$ and hence \eqref{3} implies 
\begin{center}
    ${\rm G}\mathfrak{Der}(\mathfrak L)={\rm Q}\mathfrak{Der}(\mathfrak L).$

\end{center}

\begin{definition}[see \cite{delta,deltu}]
Let $\delta$ be a fixed complex number.
An algebra $({\rm P}, \cdot, [\cdot,\cdot])$ is defined to be a {transposed $\delta$-Poisson algebra}, 
if $({\rm P}, \cdot)$ is a  commutative associative    algebra,
$({\rm P},  [\cdot,\cdot])$ is a Lie algebra and the following
identity holds:
\begin{equation}\label{deltranspois}
x\cdot [  y, z ] = \delta \big(  [x \cdot y,z ] + [y,x\cdot z]  \big).
\end{equation}
A Lie algebra $({\mathfrak L},  [\cdot, \cdot])$ has a nontivial transposed $\delta$-Poisson structure $\cdot$ if   $({\mathfrak L}, \cdot, [\cdot, \cdot])$ is a transposed $\delta$-Poisson algebra with both nonzero multiplications.
\end{definition}

Description of $\delta$-derivations (or, more generally, $\delta$-biderivations) gives a constructive way to obtain a classification of transposed $\delta$-Poisson structures on a certain Lie algebra. Namely, the next obvious statement is a generalization of \cite{FKL}.

\begin{lemma}\label{glavlem}

Let $({\mathfrak L},\cdot,[\cdot,\cdot])$ be a transposed $\delta$-Poisson algebra 
and $z$ an arbitrary element from ${\mathfrak L}.$
Then the right multiplication ${\rm R}_z$ in the associative commutative algebra $({\mathfrak L},\cdot)$ gives a $\delta$-derivation of the Lie algebra $({\mathfrak L}, [\cdot,\cdot]).$
In particular, 
 if $\mathfrak{Der}_{\delta}({\mathfrak L})\subseteq \big\langle {\rm Id}\big\rangle,$  then are no nontrivial transposed $\delta$-Poisson structures$.$

\end{lemma}

\section{Quasi-derivations of Witt algebra}\label{witt}
 
\begin{definition} 
  The Witt algebra ${\rm W}$ is an algebra   with a basis $\big\{L_i\big\}_{i\in\mathbb{Z}}$ and the multiplication given by the following way $$[L_i, L_j]=(i- j)L_{i+j} .$$
\end{definition}

\begin{proposition}\label{exquasiW}
Let   ${\rm W}$ be the Witt algebra,  then 
\begin{enumerate}
    \item[$({\mathfrak a})$] $\mathfrak{Der}\big({\rm W}\big) \ = \ {\rm Inn}\mathfrak{Der}\big({\rm W}\big)\  =\  \big\langle d_j\big \rangle_{j \in {\mathbb Z}},$
     where $d_j(L_i)\ =\  [L_j,L_i].$ 
    
    \item[$({\mathfrak b})$] $\mathfrak{Der}_{\frac{1}{2}}\big({\rm W}\big) \ =\  \langle \varphi_j \rangle_{j \in \mathbb Z},$ where 
    $\varphi_j(L_i)\ =\    L_{i+j}.$


    
\end{enumerate}

\end{proposition}
 
\begin{proof}
    ($\mathfrak{a}$) follows from \cite{Ayupov}.
    ($\mathfrak{b}$) follows from \cite{FKL}.
\end{proof}

\begin{theorem}\label{qWitt}
 $ {\rm Q}\mathfrak{Der} \big({\rm W} \big)\ =\ 
  \mathfrak{Der} \big({\rm W} \big) \oplus 
  \mathfrak{Der}_{\frac 12} \big({\rm W} \big).$
\end{theorem}
\begin{proof}
Let $f \in  {\rm Q}\mathfrak{Der}({\rm W}),$ 
hence  $f(L_i) = \sum\limits_{k\in\mathbb{Z}}\alpha_{i,k}L_k.$ 
Let 
\begin{center}
    $d=\sum\limits_{j \in {\mathbb Z}} \beta_j d_{j}$ be a derivation 
and   $\varphi=\sum\limits_{j \in {\mathbb Z}} \alpha_j \varphi_{j}$ be a $\frac 1 2$-derivation. 
\end{center}
\noindent Since the sum of quasi-derivations is again a quasi-derivation, we may replace $f$ by $f-d-\varphi,$ where 
\begin{center}
    $\alpha_j = (1-j)\alpha_{0,j} +j\alpha_{1,j+1}$ and 
$\beta_j = \alpha_{0,j} - \alpha_{1,j+1}.$
\end{center}
Then $f(L_0) = f(L_1)=0$ and we may consider a quasi-derivation $f,$ such that $\alpha_{1,j}=\alpha_{0,j}=0$ for all $j \in  {\mathbb Z}.$

Taking the related linear map $f'$ as $f'(L_i) = \sum\limits_{k\in\mathbb{Z}}\alpha'_{i,k}L_k,$ we have  
\begin{equation}\label{4}
    (k-2j)\alpha_{i,k-j}+(2i-k)\alpha_{j,k-i}-(i-j)\alpha'_{i+j,k}=0.
\end{equation}

Taking $j=0,$ $k=i$ in  \eqref{4}, we have 
$i\alpha_{i,i}\ =  i\alpha'_{i,i},$ which gives $\alpha'_{i,i}\ =  \alpha_{i,i}$ for $i\neq 0.$

Putting $j=1,$ $k=i+1$ in  \eqref{4}, we have 
$(i-1)\alpha_{i,i}\ =  (i-1)\alpha'_{i+1,i+1},$ which implies $\alpha'_{i+1,i+1}\ =  \alpha_{i,i}$ for $i\neq 1.$
Thus, we get $$\alpha'_{1,1} = 0, \qquad  \alpha'_{i,i} =\alpha_{i,i} = \alpha_{2,2}, \quad i \geq 2, \qquad \alpha'_{i,i} =\alpha_{i-1,i-1} = \alpha_{-1,-1}, \quad i \leq 0.$$

Now, taking $i=2,$ $j=3,$ $k=5$ and  $i=-2,$ $j=-3,$ $k=-5,$ respectively, we obtain $\alpha_{2,2}=0$ and $\alpha_{-1,-1}=0.$ 
Therefore, we get $$\alpha'_{i,i}\ =  \alpha_{i,i} = 0 \quad \text{for all} \quad i \in \mathbb{Z}.$$

Next taking $j=0$ in  \eqref{4}, we obtain $k\alpha_{i,k}= i\alpha'_{i,k},$ and the equation \eqref{4} has the form 
\begin{equation}\label{6}
   (i+j)\left((k-2j)\alpha_{i,k-j}+(2i-k)\alpha_{j,k-i}\right)\ = \ (i-j) k\alpha_{i+j,k} \quad \mbox{for} \quad j\neq -i.
\end{equation}

From   \eqref{6}, for $k=0,$ we have 
\begin{equation}\label{7}
   i\alpha_{j,-i}\ = \ j\alpha_{i,-j} \quad \mbox{for} \quad j\neq -i.
\end{equation}
In addition, taking $j=1,$ we derive $\alpha_{i,-1}=0$ for any $i \in \mathbb{Z}.$

Taking $k=i-1$ in \eqref{6}, we have 
\begin{equation}\label{8}(i - 2 j -1) (i + j) \alpha_{i, i-j - 1} - (i-1) (i - j) \alpha_{
   i + j, i-1}=0.\end{equation}

Now, in \eqref{6}, substituting the triple $(i,j,k):=(-i+1, j, -i),$ we obtain 
 \begin{equation}\label{9}(i+2 j) (j - i+1) \alpha_{1 - i, -i - j} + (i+j-1) i \alpha_{- i+j+1, -i}=0.\end{equation}

From \eqref{7}, we have $i \alpha_{- i + j+1, -i} = (- i + j + 1) \alpha_{i, i - j -1 }$ and $(i + j) \alpha_{1 - i, -i - j} =  
  (1 - i)  \alpha_{i + j, i-1}.$
Using these two relations in \eqref{8} and \eqref{9}, we obtain 
$$ \alpha_{i, i-j - 1} =  \alpha_{
   i + j, i-1}=0.$$

Therefore, we obtain that $f=0.$
\end{proof}

The present theorem, Lemma \ref{glavlem}, and the description of symmetric biderivations of ${\rm W}$ (see, for example, in \cite{tang2}), give the following corollary.

\begin{corollary}
If ${\rm W}$ admits a nontrivial transposed $\delta$-Poisson structure, then $\delta=\frac{1}{2}$ and 
all of transposed $\frac{1}{2}$-Poisson structures on ${\rm W}$   are completely described in {\rm \cite{FKL}}.    
\end{corollary}

\subsection{Corollaries}

The present subsection is dedicated to some results for Lie-admissible algebras with the underlying  Witt algebra. 
It is known that if $d$ is a derivation (resp., $\delta$-derivations, quasi-derivation, generalized derivation, element of centroid, etc.) of an algebra $({\rm A}, \circ),$ then it is the same for the commutator algebra $({\rm A}, [\cdot,\cdot])$ with the multiplication $[x,y]= x \circ y - y \circ x.$
Hence, the description of linear mappings of the Witt algebra gives a way to describe similar linear mappings of Lie-admissible algebras with the underlying Witt algebra.
On the other hand, the description of Lie derivations and Lie quasi-derivations of the below-mentioned Novikov-Witt and admissible Novikov-Witt algebras are completely determined in Theorem \ref{qWitt}.
The first obvious observations are as follows.

\begin{proposition}
    Let ${\rm A}$ be a Lie-admissible algebra with the underlying Witt algebra, 
    then 
    \begin{center}
        $\rm{QC} \big({\rm A}\big) \ = \ \big\langle {\rm Id} \big\rangle$ and $\mathfrak{Der}_{\delta}\big({\rm A}\big)=0$ for $\delta \notin \big\{\frac 12, 1 \big\}.$
    \end{center}
\end{proposition}

Below, we consider quasi-derivations for three types of more important Lie admissible algebras with the underlying Witt algebra, described previously in \cite{TB,KCB,BG}.

\subsubsection{Quasi-derivations of   Novikov-Witt algebras}\label{n-witt}

\begin{definition}[see  \cite{TB,KCB}]
Let $\xi \in \mathbb C,$
$\mu \in \mathbb{C},$
$\theta \in \mathbb{Z}\setminus \{0\},$
then  the Novikov-Witt algebra ${\rm W}(\xi,\mu, \theta)$ is an algebra  with a basis $\big\{W_i\big\}_{i\in\mathbb{Z}}$ and the multiplication given by the following way 
  $$ W_n \circ  W_m =(\xi +m)W_{n+m} + \mu W_{n+m+\theta}.$$
\end{definition} 

\begin{proposition}
Let   ${\rm W}(\xi,\mu, \theta)$ be the   Novikov-Witt algebra,  then 
\begin{enumerate}
    \item[$({\mathfrak a}1)$] if $\xi \notin \mathbb Z,$ then $\mathfrak{Der}\big({\rm W}(\xi,0, \theta)\big) = \big\langle {\mathfrak D}_1\big\rangle,$
    where ${\mathfrak D}_1(W_n)= n W_n.$ 

    \item[$({\mathfrak a}2)$] if   $\xi \in \mathbb Z,$ then $\mathfrak{Der}\big({\rm W}(\xi,0, \theta)\big) =\big\langle {\mathfrak D}_1, {\mathfrak D}_2\big\rangle,$
    where 
    ${\mathfrak D}_2(W_n)= \big(n+\xi \big) W_{n-\xi}.$
 
     \item[$({\mathfrak a}3)$]  $\mathfrak{Der} \big({\rm W}(\xi,\mu \neq 0, \theta)\big)=0.$

    \item[$({\mathfrak b})$] 
    $\mathfrak{Der}_{\frac{1}{2}}\big({\rm W}(\xi,\mu, \theta)\big)=\big\langle {\rm Id}\big \rangle.$

      \item[$({\mathfrak c}1)$] 
if $j\notin \big\{0, -\xi \big\},$ then         ${\mathfrak F}_{j}(W_n)=  (n+\xi)  W_{n+j}$
is a quasi-derivation with a related mapping ${\mathfrak F}_j'$ defined by 
  \begin{center}
        ${\mathfrak F}_j'(W_n)= {(n+j+2 \xi)}   W_{n+j}.$  \end{center} 
 The set of such quasi-derivations is denoted by       
 ${\rm Q}\mathfrak{Der}^*\big( {\rm W}(\xi, 0, \theta)  \big).$

    \item[$({\mathfrak c}2)$] 
       For an arbitrary set of elements     $\{ \kappa_k \}_{ k\in {\mathbb Z} }$  with only a finite number of nonzero elements, the linear map ${\mathfrak F}_{\{ \kappa_k \}}$ defined as    
  \begin{center}
         ${\mathfrak F}_{\{ \kappa_k \}}(W_n)=\sum\limits_{k \in  {\mathbb Z}  }
             \big( \kappa_{k-\theta}\mu+ \kappa_k {( \xi + n)} \big)  W_{n+k},$
  \end{center}
 is a quasi-derivations with a related mapping ${\mathfrak F}_{\{ \kappa_k \}}'$ defined by 
  \begin{center}
        ${\mathfrak F}_{\{ \kappa_k \}}'(W_n)=\sum\limits_{k \in  {\mathbb Z}  }
       \big( 2 \kappa_{k-\theta}\mu +\kappa_k {( 2\xi +k+ n)}  \big) W_{n+k}.$  \end{center}

The set of such quasi-derivations is denoted by 
 ${\rm Q}\mathfrak{Der}^*\big( {\rm W}(\xi,\mu\neq 0, \theta) \big).$
 
    \item[$({\mathfrak d})$] 
       $ {\rm Q}\mathfrak{Der}\big({\rm W}(\xi,\mu, \theta) \big)\  =\  
        \mathfrak{Der}\big( {\rm W}(\xi,\mu, \theta) \big) \oplus 
        \mathfrak{Der}_{\frac{1}{2}}\big( {\rm W}(\xi,\mu, \theta) \big) \oplus 
       {\rm Q}\mathfrak{Der}^*\big( {\rm W}(\xi,\mu, \theta) \big).$ 
 \end{enumerate}
     
\end{proposition}
\begin{proof}
Let $f \in {\rm Q}\mathfrak{Der} \big({\rm W}(\xi,\mu, \theta)\big),$ 
then $f \in {\rm Q}\mathfrak{Der} \big({\rm W}\big),$ 
and thanks to Theorem    \ref{qWitt},
there is a finite set $\{ \alpha_k, \beta_k  \}_{k \in \mathbb Z},$ such that 
\begin{longtable}{lcl}
    $f(W_n) $&$=$&$ \sum\limits_{k \in \mathbb Z}\big ( \alpha_k (k-n)+ \ \  \beta_k \big) W_{n+k},$\\
    $f'(W_n) $&$=$&$ \sum\limits_{k \in \mathbb Z}\big ( \alpha_k (k-n)+2 \beta_k \big) W_{n+k}.$
 \end{longtable}\noindent 
On the other side, 

\begin{center}
    $ f(W_n) \circ W_m + W_n \circ f(W_m) \ = \ (\xi+m) f'(W_{n+m}) +\mu f'(W_{n+m+\theta}),$ i.e.,
\end{center}

\begin{flushleft}
$\big( 
\alpha_k(k-n) +\beta_k  \big) \big(\xi +m\big)
+\big( \alpha_{k-\theta} (k-\theta -n)+\beta_{k-\theta}   \big)\mu+$
\end{flushleft}
\begin{center}
    $\big( \alpha_k(k-m)+\beta_k \big)\big( \xi+k+m\big) +\big(\alpha_{k-\theta}(k-\theta-m) +\beta_{k-\theta}  \big)\mu\ =  $
\end{center}
\begin{flushright}
$\big(\xi+m\big) \big( \alpha_k(k-n-m)+2\beta_k \big)+\big(\alpha_{k-\theta} (k-n-m-2\theta) +2\beta_{k-\theta}   \big) \mu.$
\end{flushright}
The last gives 
\begin{center}
$k \big(k \alpha_k + \beta_k + \alpha_{k-\theta} \mu  + \alpha_k \xi  \big) \ = \ 0$ and   if $k\neq0,$ then $\beta_k=   -\alpha_{k-\theta}   \mu  -\alpha_k  \xi-k \alpha_k .$
\end{center}

 Below, we consider our particular cases. 

\begin{enumerate}
    \item[$(\mathfrak{a})$] 
If $f \in   \mathfrak{Der}\big({\rm W}(\xi,\mu, \theta)\big),$ then $\beta_k=0$
and 
\begin{enumerate}
        \item[$(\mathfrak{a}1)$] 
 if $\mu=0$ and  $\xi \notin \mathbb{Z},$ then $\alpha_k=0$ for $k\neq 0,$
 which gives the part $(\mathfrak{a}1)$ from our statement.

        \item[$(\mathfrak{a}2)$] 
 if $\mu=0$ and  $\xi \in \mathbb{Z},$ then $\alpha_k=0$ for $k\notin \big\{ 0, -\xi \big\},$
 which gives the part $(\mathfrak{a}2)$ from our statement.

      \item[$(\mathfrak{a}3)$] 
 if $\mu\neq 0,$ then $\alpha_{k-\theta} = -\alpha_k(\xi+k)\mu^{-1},$
 i.e., if there is $\alpha_k\neq0,$ then we have an infinite set of nonzero elements  $\{\alpha_j\}_{j \in \mathbb{Z}},$ which gives a contradiction with the finitary of the set  $\{\alpha_k\}_{k \in \mathbb{Z}}.$ 

\end{enumerate}

  \item[$(\mathfrak{b})$] 
If $f \in   \mathfrak{Der}_{\frac{1}{2}} \big({\rm W}(\xi,\mu, \theta)\big),$ 
then $\alpha_k=0$
and  $\beta_k=0$ for $k\neq 0,$ which gives  the   part $(\mathfrak{b})$ from our statement.

\item[$(\mathfrak{c})$] 
The mappings from the part $(\mathfrak{c})$ give  quasi-derivations by some direct calculations.

\item[$(\mathfrak{d})$]  
Using 
$\beta_k= -k \alpha_k  - \alpha_{k-\theta}  \mu -\alpha_k  \xi$
and taking $\kappa_k=  -\alpha_k,$  we have 
\begin{longtable}{lcr}   
 $f(W_n)$&$=$&$\sum\limits_{k \in  {\mathbb Z}  }
             \big( \kappa_{k-\theta}\mu+ \kappa_k {( \xi + n)} \big)  W_{n+k},$\\
 $f'(W_n)$&$=$&$\sum\limits_{k \in  {\mathbb Z}  }
       \big( 2 \kappa_{k-\theta}\mu +\kappa_k {( 2\xi +k+ n)}  \big) W_{n+k}.$ 
\end{longtable}
Which concludes our statement.
\end{enumerate}   
\end{proof}

\subsubsection{Quasi-derivations of admissible Novikov-Witt algebras}\label{adn-witt}

\begin{definition}[see  \cite{BG}]
Let $\gamma  \in \mathbb C,$ then  the admissible Novikov-Witt algebra ${\rm N}_{\gamma}$ is an algebra  with a basis $\big\{W_i\big\}_{i\in\mathbb{Z}}$ and the multiplication given by the following way 
  $$ W_n \circ  W_m =(\gamma +m+2n)W_{n+m} .$$
\end{definition}

\begin{proposition}\label{exquasiadNW}
Let   ${\rm N}_{\gamma}$ be the admissible Novikov-Witt algebra,  then 
\begin{enumerate}
    \item[$({\mathfrak a}1)$] if $\gamma \notin 3\mathbb Z,$ then $\mathfrak{Der} \big({\rm {\rm N}_{\gamma}}\big) =\big\langle {\mathfrak D}_1\big\rangle,$
    where ${\mathfrak D}_1(W_n)= n W_n.$ 

    \item[$({\mathfrak a}2)$] if $\gamma \in 3\mathbb Z,$ then $\mathfrak{Der}\big({\rm {\rm N}_{\gamma}}\big) =\big\langle {\mathfrak D}_1, {\mathfrak D}_2\big\rangle,$
    where
    ${\mathfrak D}_2(W_n)= \big(n+\frac{\gamma}{3} \big) W_{n-\frac{\gamma}{3}}.$
    
    \item[$({\mathfrak b})$] 
    $\mathfrak{Der}_{\frac{1}{2}} \big({\rm {\rm N}_{\gamma}}\big)=\big\langle {\rm Id}\big\rangle.$ 
    
    \item[$({\mathfrak c})$] 
if $j \notin \big\{0, -\frac{\gamma}{3} \big\},$ then      ${\mathfrak F}_{j}(W_n)=  \big(n+\frac{\gamma} 3\big)  W_{n+j}$ 
  is a quasi-derivation with a related mapping ${\mathfrak F}_j'$ defined by 
  \begin{center}
        ${\mathfrak F}_j'(W_n)= {\big(n+j+\frac{2 \gamma} 3\big)}   W_{n+j}.$  \end{center}
The set of such quasi-derivations is denoted by
${\rm Q}\mathfrak{Der}^*\big({\rm {\rm N}_{\gamma}}\big).$

       \item[$({\mathfrak d})$] 
       $ {\rm Q}\mathfrak{Der}\big({\rm {\rm N}_{\gamma}}\big) \ = \  
        \mathfrak{Der}\big({\rm {\rm N}_{\gamma}}\big) \oplus 
        \mathfrak{Der}_{\frac{1}{2}}\big({\rm {\rm N}_{\gamma}}\big) \oplus 
       {\rm Q}\mathfrak{Der}^*\big({\rm {\rm N}_{\gamma}}\big).$ 
\end{enumerate}

\begin{proof}
    Let $f \in {\rm Q}\mathfrak{Der} \big({\rm {\rm N}_{\gamma}}\big),$ 
    then  $f \in {\rm Q}\mathfrak{Der}\big({\rm {\rm W}}\big),$
and thanks to Theorem    \ref{qWitt},
there is a finite set $\{ \alpha_k, \beta_k \}_{k \in \mathbb Z},$ such that 
\begin{longtable}{lcl}
    $f(W_n) $&$=$&$ \sum\limits_{k \in \mathbb Z}\big ( \alpha_k (k-n)+ \ \ \beta_k \big) W_{k+n},$\\
    $f'(W_n) $&$=$&$ \sum\limits_{k \in \mathbb Z}\big ( \alpha_k (k-n)+2 \beta_k \big) W_{k+n}.$
 \end{longtable}\noindent 
On the other side, 

\begin{center}
    $ f(W_n) \circ W_m + W_n \circ f(W_m) \ = \ (\gamma+m+2n) f'(W_{n+m}),$ i.e.,
\end{center}

\begin{flushleft}
$\big(\gamma + 2n +2k+m\big)\big(\alpha_k(k-n)+\beta_k \big)+
\big(\gamma+2n+m+k\big)\big(\alpha_k(k-m)+\beta_k \big)\ =$ \end{flushleft}
\begin{flushright}
$\big(\gamma+m+2n\big)\big(\alpha_k(k-m-n)+2\beta_k \big).$
\end{flushright}
The last gives 
\begin{center}
$k \big(\alpha_k(3k+\gamma) +3\beta_k \big)\ = \ 0$ and 
if $k\neq0,$ then $\beta_k=-\alpha_k \big(\frac{\gamma}{3} + k \big).$
\end{center}

Below, we consider our particular cases. 

\begin{enumerate}
    \item[$(\mathfrak{a})$] 
If $f \in   \mathfrak{Der} \big({\rm {\rm N}_{\gamma}}\big),$ then $\beta_k= 0$
and 
\begin{enumerate}
        \item[$(\mathfrak{a}1)$] 
 if $\gamma \notin 3\mathbb{Z},$ then $\alpha_k=0$ for $k\neq 0,$
 which gives the part $(\mathfrak{a}1)$ from our statement.

        \item[$(\mathfrak{a}2)$] 
 if $\gamma \in 3\mathbb{Z},$ then $\alpha_k=0$ for $k\notin \big\{ 0, -\frac{\gamma}{3} \big\},$
 which gives the part $(\mathfrak{a}2)$ from our statement.

\end{enumerate}

  \item[$(\mathfrak{b})$] 
If $f \in   \mathfrak{Der}_{\frac{1}{2}} \big({\rm {\rm N}_{\gamma}}\big),$ then $\alpha_k= 0$
and  $\beta_k=0$ for $k\neq 0,$ which gives  the   part $(\mathfrak{b})$ from our statement.

\item[$(\mathfrak{c})$] 
The mappings from the part $(\mathfrak{c})$ give quasi-derivations by some direct calculations.

\item[$(\mathfrak{d})$] 
Let us remember that the sum of a quasi-derivation and a derivation or a $\frac{1}{2}$-derivation gives another quasi-derivation.
Hence, we can replace $f$ by $f- \alpha_0 \mathfrak{D}_1 - \beta_0 {\rm Id},$
i.e. consider $f$ with $\alpha_0=\beta_0= 0.$
Using\begin{center}
     $\beta_k=-\alpha_k \big(\frac{\gamma}{3} +k \big)$
and taking $\kappa_k= - \frac{\alpha_k}{3},$
\end{center}  we have 
\begin{longtable}{lcl}   
$f(W_n)$&$=$&$\sum\limits_{k \in  {\mathbb Z} \setminus \{0\} }
              \kappa_k (\gamma +3 n)  W_{k+n},$\\
$f'(W_n)$&$=$&$\sum\limits_{k \in  {\mathbb Z} \setminus \{0\} }\kappa_k {(2 \gamma + 3 k + 3n)}   W_{n+k}.$  
\end{longtable}
Which concludes our statement.
\end{enumerate}

\end{proof}
    
\end{proposition}
 
\section{Quasi-derivations of Virasoro algebra}\label{vira}

\begin{definition}
The Virasoro algebra $\bold{Vir}$
 is an algebra with a basis $\big\{L_i, C\big\}_{i\in\mathbb{Z}}$
 and the multiplication table given in the following way 
 $$[L_i,L_j]=(i-j)L_{i+j}+{(i^3-i)} \delta_{i+j,0}C.$$
\end{definition}

\begin{proposition}\label{exquasiV}
Let   $\bold{Vir}$ be the Virasoro algebra,  then 
\begin{enumerate}
    \item[$({\mathfrak a})$] $ \mathfrak{Der}\big(\bold{Vir}\big)= {\rm Inn}\mathfrak{Der}\big(\bold{Vir}\big)=\big\langle d_j\big\rangle_{j \in {\mathbb Z}},$
     where $d_j(L_i)= [L_j,L_i].$

    \item[$({\mathfrak b})$] $\mathfrak{Der}_{\frac{1}{2}} \big(\bold{Vir}\big) = \big \langle {\rm Id} \big\rangle.$
    
    \item[$({\mathfrak c})$] 
    $ {\rm Ann}^Q  \big( {\bold {Vir}} \big)=\big\{ g \in {\rm End}({\bold {Vir}}) \ |\  g({ \bold {Vir}}) \subseteq  {\rm Ann}   \big( {\bold {Vir}} \big) = \big\langle C \big\rangle \big\} \subseteq {\rm Q}\mathfrak{Der} \big( \bold{Vir} \big).$         
\end{enumerate}

\end{proposition}

\begin{proof}
    ($\mathfrak{a}$) follows from \cite{Ayupov}.
    ($\mathfrak{b}$) follows from \cite{FKL}.
    ($\mathfrak{c}$) follows from direct computations.

\end{proof}

\begin{theorem}
 $ {\rm Q}\mathfrak{Der} \big( \bold{Vir} \big) =
  \mathfrak{Der} \big( \bold{Vir} \big) \oplus 
  \mathfrak{Der}_{\frac 12} \big( \bold{Vir} \big) \oplus 
  \rm{Ann}^Q  \big( \bold{Vir} \big).$

\end{theorem}
\begin{proof}
Let $f \in  {\rm Q}\mathfrak{Der} \big(\bold{Vir} \big),$ 
put  $$f(L_i) = \sum\limits_{k\in\mathbb{Z}}\alpha_{i,k}L_k+\beta_iC, \quad   f(C) = \sum\limits_{k\in\mathbb{Z}}\alpha_{k}L_k+\beta C.$$ 

Similar to the case of the Witt algebra (see Theorem \ref{qWitt}), 
we can replace $f$ by $f-\varphi-d,$
where \begin{center}
  $\varphi=\alpha_{0,0}\rm{Id}$ and 
$d= \big(\alpha_{0,0}-\alpha_{1,1}\big)d_0+\sum\limits_{k\in\mathbb{Z}\setminus{ \{0\}}}\frac{\alpha_{0,k}}{k}d_k.$
\end{center}

Now we considering a quasi-derivation $f,$ such that $\alpha_{1,1}=\alpha_{0,j}=0$ for all $j\in\mathbb{Z}.$ 

\medskip

Taking the related linear map $f'$ as $$f'(L_i) = \sum\limits_{k\in\mathbb{Z}}\alpha'_{i,k}L_k+\beta'_iC, \quad f'(C) = \sum\limits_{k\in\mathbb{Z}}\alpha'_{k}L_k+\beta' C,$$   we have  
\begin{equation}\label{1a}
\begin{array}{cc}
    \sum\limits_{ k\in\mathbb{Z}}\Big(\alpha_{i,k}\left((k-j)L_{j+k} + \delta_{j+k,0}{(k^3-k)}C\right)+\alpha_{j,k}\left((i-k) L_{i+k} + \delta_{i+k,0}{(i^3-i)} C\right)\Big)\ =\ \\[2mm] 
    (i-j)\big(\sum\limits_{ k\in\mathbb{Z}}\alpha'_{i+j,k}L_k+\beta'_{i+j}C\big)+
    \delta_{i+j,0}{(i^3-i)}\big(\sum\limits_{ k\in\mathbb{Z}}\alpha'_{k}L_k+\beta'C\big).
    \end{array}
\end{equation}
By comparing the coefficients of the basis in \eqref{1a} we get
\begin{equation}\label{2a}
  (k-2j)\alpha_{i,k-j}+(2i-k)\alpha_{j,k-i}-(i-j)\alpha'_{i+j,k}-\delta_{i+j,0}{(i^3-i)}\alpha'_{k}\ =\ 0,
\end{equation}
\begin{equation}\label{3a}
  -{(j^3-j)} \alpha_{i,-j}+{(i^3-i)} \alpha_{j,-i}-(i-j)\beta'_{i+j}-\delta_{i+j,0}{(i^3-i)}\beta'\ =\ 0.
\end{equation}
Also, from the fact that $[L_i,C] = 0,$ we obtain
\begin{center}
    $0\ =\ [L_i,f(C)]\ =\ [L_i,\sum\limits_{k\in\mathbb{Z}}\alpha_{k}L_k]\ =\ \sum\limits_{k\in\mathbb{Z}}\alpha_{k}(i-k)L_{i+k},$
\end{center}
therefore $\alpha_k=0$ for all $k\in\mathbb{Z}.$

From \eqref{2a} we get $\alpha'_{i,i}=\alpha_{i,i}=0,$ $i\neq 0$ for the case $j\neq -i.$

In the following step, the value of the coefficients $\beta'$ and   $\beta'_i,$  $i\in\mathbb{Z}$
 is determined using the equation \eqref{3a}. 
 Letting $j=0,$ $i\neq0$ in \eqref{3a}, it gives  $\beta'_{i}=0,$  for $i\neq 0.$ Putting $i=1,$ $j=-1$ in  \eqref{3a}, we get $\beta'_0=0$ and putting $j=-i,$ 
we have  $\beta'=0.$
  Using the relation \eqref{7} 
in \eqref{3a}, we derive $j(i^2-j^2)\alpha_{i,-j}=0.$ Then it implies 
$\alpha_{i,j}=0$ for $j\neq -i .$
Putting $j=0$ in \eqref{2a}, we obtain $\alpha'_{i,k}=0$ for $i\neq 0,$ $k\neq - i.$ Letting $k=j-i,$ $j\neq -i$ in \eqref{2a}, we obtain $\alpha_{i,-i}=0$ for $i\in\mathbb{Z}$ therefore it implies   $\alpha'_{i,-i}=0$  for  $i\in\mathbb{Z}\setminus{0}.$

Next, we analyze the case  $j = -i.$  To do this, if we substitute $-i$ for $j$ in \eqref{2a}, the following relation arises:
\begin{equation}\label{31e}
 (k+2i) \alpha_{i,k+i}+(2i-k)\alpha_{-i,k-i}-2i\alpha'_{0,k}-{(i^3-i)} \alpha'_k\ =\ 0.
\end{equation} 
 We have 
\begin{equation}\label{32e}
2\alpha'_{0,k}+{(i^2-1)} \alpha'_k\ =\ 0,\quad i\neq0.
\end{equation}  
  Putting $i=1$ and $i\neq \pm 1$  in \eqref{31e}, we obtain 
$\alpha'_{0,k}=\alpha'_k=0$  for all $k\in\mathbb{Z}.$

Summarizing, our mappings $f$ and $f'$ have the following form
$$f(L_i)=\beta_iC, \quad  f(C)=\beta C,\quad  f'(L_i)=f'(C)=0.$$

Hence,  $f$ can be represented as a sum of mappings from 
Proposition \ref{exquasiV}.  
\end{proof}

The present theorem, Lemma \ref{glavlem}, and a well-known fact that each  biderivation of $\bold{Vir}$ is inner (i.e., skew-symmetric) give the following corollary.

\begin{corollary}
$\bold{Vir}$ does not admit nontrivial transposed $\delta$-Poisson structures.    
\end{corollary}

\section{Quasi-derivations of Lie algebras ${\mathcal W}(a,b)$} \label{wab}

\begin{definition}[see \cite{kac, mat}]
The  algebra ${\mathcal W}(a,b)$ with $a,b\in\mathbb{C}$ is an algebra with a basis $\big\{L_i, I_i\big\}_{i\in\mathbb{Z}}$
 and the multiplication table is given in the following way 
$$[L_i,L_j]=(i-j)L_{i+j}, \quad [L_i,I_j]=-(j+bi+a)I_{i+j}.$$

\noindent
${\rm V}(a,b) = \big \langle I_i \big \rangle_{i \in \mathbb Z}$ is called the tensor density module of $\rm W.$

\end{definition}

 \begin{lemma}[see  \cite{gao}]\label{lemma}
 ${\mathcal {W}}(a,b) \simeq {\mathcal {W}}(a+k,b)$ for any $k\in\mathbb{Z}.$
  Up to isomorphism, the center of ${\mathcal {W}}(a,b)$ is
\begin{equation*}
    {\rm Ann}\big(\mathcal {W}(a,b)\big) = 
   \begin{cases} 
       \mathbb{C}I_0, & \text{if } (a,b) = (0,0) \\
      0, & \text{otherwise}
    \end{cases}
   \end{equation*}
\end{lemma}

\begin{proposition} \label{lemma11} 
For the algebra  ${\mathcal W}(a,b)$ the following statements are true:
\begin{enumerate}
    \item[$(\mathfrak{a})$] The derivations of ${\mathcal W}(a,b)$ are determined by the following:
 \begin{equation*}
    \mathfrak{Der} \big(\mathcal {W}(a,b) \big) = 
    \begin{cases} 
        \operatorname{Inn}\big(\mathcal {W}(a,b)\big)\oplus \big\langle  D_1 \big \rangle \oplus\big\langle D_2^{0}\big \rangle\oplus\big\langle D_3\big \rangle &  (a,b) = (0,0), \\
         \operatorname{Inn}\big(\mathcal {W}(a,b)\big)\oplus\big\langle D_1\big \rangle\oplus\big\langle D_2^{1} \big \rangle&  (a,b) =(0,1),\\
        \operatorname{Inn}\big(\mathcal {W}(a,b)\big)\oplus\big\langle D_1\big \rangle\oplus\big\langle D_2^{2} \big \rangle&  (a,b) = (0,2),\\
         \operatorname{Inn}\big(\mathcal {W}(a,b)\big)\oplus\big\langle D_1 \big \rangle&  \text{otherwise},
    \end{cases}
   \end{equation*}   \noindent 
where the derivations $D_1,$ $D_2^{0},$ $D_2^{1},$ $D_2^{2},$ $D_3$ are defined as follows for all $i\in \mathbb{Z},$
\begin{longtable}{rcrrcl}
     $D_1(L_i)$&$=$&$0,$&   $D_1(I_i)$&$=$&$I_i;$\\
     $D_2^{k}(L_i)$&$=$&$i^{k+1}I_i,$&    $D_2^{k}(I_i)$&$=$&$0, \ k \in \big\{0,1,2\big\};$\\

     $D_3(L_i)$&$=$&$I_i,$&   $D_3(I_i)$&$=$&$0.$
\end{longtable}

\item[$(\mathfrak{b1})$] 
If $b\neq-1,$ then $  \mathfrak{Der}_{\frac{1}{2}} \big({\mathcal W}(a,b)\big) = \big\langle {\rm Id} \big\rangle.$

\item[$(\mathfrak{b2})$] 
$  \mathfrak{Der}_{\frac{1}{2}} \big({\mathcal W}(a,-1) \big) = \big\langle \varphi_j, \psi_j \big\rangle_{j \in \mathbb Z},$
where 
\begin{center}
    $\varphi_j(L_i)=  L_{i+j}, \ 
\varphi_j(I_i)= I_{i+j}; \quad  
\psi_j(L_i)= I_{i+j}, \
\psi_j(I_i)= 0.$
\end{center}

\item[$(\mathfrak{c})$] 
The following linear mappings   given by   
\begin{longtable}{|rclrcllll|}
\hline
${\mathfrak F}_j(L_i) $&$=$&$  i L_{i+j},$ & ${\mathfrak F}_j(I_i)$&$ = $& \multicolumn{4}{l|}{$ (a+i) I_{i+j}$ }\\
\hline
 
$g_j(L_i)$&$=$&$i I_{i+j},$  & $g_j(I_i)$&$=$&$0$ &&&\\
\hline
$h(L_i)$&$=$&$I_{i},$  & $h(I_i)$&$=$&$0$ &&&\\ 
\hline
${\mathfrak p}_{0,j}(L_i)$ & $=$ & $\frac{i(i-1)(a+2+j)}{2(a+i+j)}I_{i+j}$, & ${\mathfrak p}_{0,j}(I_i)$ & $=$ & $0$ & \text{for}&  $a\notin \mathbb{Z},$  &$b=0$\\
\hline
${\mathfrak p}_{1,j}(L_i)$ & $=$ & $\frac{i(i-1)}{2}I_{i+j}$, & ${\mathfrak p}_{1,j}(I_i)$ & $=$ & $0$ & \text{for}& & $b=1$\\
\hline
${\mathfrak p}_{2,j}(L_i)$ & $=$ & $\frac{i(i-1)(a+1+i+j)}{2(a+3+j)}I_{i+j}$, & ${\mathfrak p}_{2,j}(I_i)$ & $=$ & $0$ & \text{for}&  $(a,j)\neq (0,-3) $ & $b=2$\\
\hline

${\mathfrak E}(L_i) $&$=$&$  \frac{i(i-1)(i-2)}{6} I_{i-3},$ & ${\mathfrak E}(I_i)$&$ = $&$ 0$ & \text{for}& $a=0$ & $b=2$\\
\hline

\multicolumn{9}{l}{}\\
\multicolumn{9}{l}{$\mbox{are quasi-derivations of ${\mathcal W}(a,b)$ with related mappings}$}\\ 

\hline
${\mathfrak F}'_j(L_i) $&$=$&$  (i+j) L_{i+j},$& ${\mathfrak F}'_j(I_i) $&$=$&\multicolumn{4}{l|}{$  (a+i+j) I_{i+j}$ }\\
\hline 
$g'_j(L_i)$&$=$&$(a+i+j) I_{i+j},$&$g'_j(I_i)$&$=$&$0$ &&&\\
\hline
$h'(L_i)$&$=$&$(1-b)I_{i},$ &  $h'(I_i)$&$=$&$0$ &&&\\ \hline
${\mathfrak p}'_{0,j}(L_i)$ & $=$ & $\frac{(i-1)(a+2+j)}{2}I_{i+j}$, & ${\mathfrak p}'_{0,j}(I_i)$ & $=$ & $0$ & \text{for}& $a\notin \mathbb{Z}$ & $b=0$ \\ \hline
${\mathfrak p}'_{1,j}(L_i)$ & $=$ & $\frac{(i-1)(a+i+j)}{2}I_{i+j}$, & ${\mathfrak p}'_{1,j}(I_i)$ & $=$ & $0$ & \text{for}& & $b=1$\\
\hline
${\mathfrak p}'_{2,j}(L_i)$ & $=$ & $\frac{(i-1)(a+i+j)(a+1+i+j)}{2(a+3+j)}I_{i+j}$, & ${\mathfrak p}'_{2,j}(I_i)$ & $=$ & $0$ & \text{for}& $(a,j)\neq (0,-3)$ & $b=2$ \\
\hline
${\mathfrak E}'(L_i) $&$=$&$  \frac{(i-1)(i-2)(i-3)}{6} I_{i-3},$ & ${\mathfrak E}'(I_i)$&$ = $&$ 0$ & \text{for}& $a=0$ & $b=2$ \\
\hline


\end{longtable}
\noindent 


 \item[$({\mathfrak d})$] 
    $ {\rm Ann}^Q  \big( {\mathcal W}(a,b) \big)=\big\{ g \in {\rm End}({\mathcal W}(a,b)) \ |\  g({\mathcal W}(a,b)) \subseteq  {\rm Ann}  \big( {\mathcal W}(a,b) \big) \big\} \subseteq {\rm Q}\mathfrak{Der} \big( {\mathcal W}(a,b) \big).$      

\end{enumerate}

\end{proposition}

\begin{proof}
    ($\mathfrak{a}$) follows from \cite{gao}.
    ($\mathfrak{b}$) follows from \cite{FKL}.
    ($\mathfrak{c})$ and $(\mathfrak{d}$) follow  from direct computations.
\end{proof}

\begin{definition}
    Let ${\mathcal V}$ be a bimodule over an algebra $({\mathfrak L}, [\cdot,\cdot])$ with an action $\circ : {\mathfrak L} \times {\mathcal V},{\mathcal V} \times {\mathfrak L} \rightarrow {\mathcal V},$
    then     
   a linear map $f: {\mathfrak L} \rightarrow \mathcal{V}$ is called a quasi-derivation if there exists  a linear map $f'  : {\mathfrak L} \rightarrow \mathcal{V}$ such that 
\begin{eqnarray}\label{qderm}
    f(x) \circ y +x \circ f(y)\ =\ f'[x,y].
\end{eqnarray}
\end{definition}

 \begin{lemma}\label{lemmabeta}
     
Let $f \in {\rm Q}\mathfrak{Der} \big({\rm   W}, {\rm V}(a,b)\big),$ then $f$ is a sum of
  mappings from  Proposition \ref{lemma11}.   
In particular, 
  $ {\rm Q}\mathfrak{Der} \big({\rm   W}, {\rm V}(a,-1)\big) =
  \mathfrak{Der} \big({\rm   W}, {\rm V}(a,-1)\big) \oplus 
  \mathfrak{Der}_{\frac 12} \big({\rm   W}, {\rm V}(a,-1)\big).$

 \end{lemma}
\begin{proof}
Let $f \in  {\rm Q}\mathfrak{Der}\big({\rm W},   {\rm V}(a,b)\big),$ then 
\begin{longtable}{rclrcl}
$f(L_i) $&$ = $&$ 
\sum\limits_{k\in\mathbb{Z}}\beta_{i,k}I_k, $& 
$f'(L_i) $&$ = $&$  \sum\limits_{k\in\mathbb{Z}} \beta'_{i,k}I_k.$
\end{longtable}
By \eqref{qderm}, we have
\begin{longtable}{rcl}
$\sum\limits_{ k\in\mathbb{Z}}\left((k+a+bj)\beta_{i,k}I_{j+k} -(k+a+bi)\beta_{j,k} I_{i+k}\right)$&$=$&$(i-j)\sum\limits_{ k\in\mathbb{Z}}\beta'_{i+j,k}I_k$
\end{longtable}
or, equivalently
\begin{eqnarray}
\label{1.2} (k-j+a+bj)\beta_{i,k-j} -(k-i+a+bi)\beta_{j,k-i}\ =&(i-j)\beta'_{i+j,k}.
\end{eqnarray}

\begin{enumerate}
\item Let $b=-1.$  
Taking  $\xi_k$ and $\eta_k$ as solutions of the following system of equations $$\xi_k+(k+a)\eta_k=\beta_{0,k}, \quad \xi_{k-1}+(k+a-2)\eta_{k-1}=\beta_{1,k},$$ for the $\frac 1 2$-derivation  $\psi(L_i)=\sum\limits_{k\in\mathbb{Z}}\xi_kI_{i+k}$ and for the 
element $x=\sum\limits_{k\in\mathbb{Z}}\eta_kI_k,$ we get that
\begin{center}
    $\psi (L_0)+{\rm ad}_x(L_0)=f(L_0)$ and 
$\psi (L_1)+{\rm ad}_x(L_1)=f(L_1).$
\end{center}

Thus, replacing $f$ by $f-\psi-{\rm ad}_x,$ we may suppose that $\beta_{0,k}=\beta_{1,k}=0$ for all $k\in\mathbb{Z}.$

Putting $i=0$ and  $i=1$ in \eqref{1.2}, we get 
\begin{center}
    $(k+a)\beta_{j,k}= j \beta'_{j,k}$ and   
$(k+a-2)\beta_{j,k-1}= (j-1) \beta'_{j+1,k},$
\end{center} which implies  
\begin{equation}\label{1.303}(j-2)(k+a)\beta_{j,k} = j(k+a-2)\beta_{j-1,k-1}.\end{equation}

\begin{enumerate}
\item Let $a\notin \mathbb{Z},$ then, 
    by induction we get 
    \begin{longtable}{rcll}
       $\beta_{j,k} $&$=$&$ \frac{j(j-1)(k+a+1-j)(k+a+2-j)}{2(k+a)(k+a-1)} \beta_{2,k-j+2}, $&$ j\geq 2;$\\ 
       $\beta_{j,k} $&$=$&$ \frac{j(j-1)(k+a-1-j)(k+a-2-j)}{2(k+a)(k+a-1)} \beta_{-1,k-j-1}, $&$ j\leq -1.$ 
    \end{longtable}
    

Now, letting $i=2,$ $j=3$ and $i=-1,$ $j=-2$ in \eqref{1.2}, respectively, we get
\begin{longtable}{lclcl}
$0$&$=$&$(k+a-6)\beta_{2,k-3}-(k+a-4) \beta_{3,k-2}+\beta'_{5,k}$&$=$&$-\frac{12}{(k+a-2) (k+a-1)}\beta_{2,k-3};$\\
$0$&$=$&$(k+a+4)\beta_{-1,k+2}-(k+a+2) \beta_{-2,k+1}-\beta'_{-3,k}$&$=$&$\frac{12}{(k+a-1) (k+a)}\beta_{-1,k+2}.$
\end{longtable}

Due to arbitrary $k$, we have that $\beta_{2,k}=0$ and $\beta_{-1,k}=0$ for all $k \in \mathbb{Z}.$ Thus, we obtain 
$\beta_{j,k}=0,$ for all $j, k \in \mathbb{Z}.$

\item Let $a\in \mathbb{Z}.$ Without loss of generality, we may assume $a=0.$ 
By taking $k=0$ and $k=2$ in equation  \eqref{1.303}, we obtain, respectively,
$\beta_{j,-1} =0$ for $j \neq -1$ and  $\beta_{j,2} =0$ for $j \neq 2.$
Then, similarly to the previous case, we obtain inductively that
\begin{longtable}{lcll}
$\beta_{j,k}$&$ =$&$ \frac{j(j-1)(k+1-j)(k+2-j)}{2k(k-1)} \beta_{2,k-j+2},$&$ \ 2 \leq j \leq  k;$\\ 
$\beta_{j,k}$&$ =$&$ \frac{j(j-1)(k-1-j)(k-2-j)}{2k(k-1)} \beta_{-1,k-j-1},$&$\  k \leq j\leq -1;$\\ 
$\beta_{j,k} $&$=$&$0,$&$ \ k \geq 2 \  \text{and} \ j< 2 \ \text{or} \ j > k;$\\
$\beta_{j,k} $&$=$&$0,$&$ \ k \leq -1 \  \text{and} \ j> -1 \ \text{or} \ j < k.$
\end{longtable}
Now, by setting $i=2,$ $j=3$ and $i=-1,$ $j=-2$ in \eqref{1.2}, respectively, we obtain
$\beta_{2,k}=0$ for $k \geq 2,$ and $\beta_{-1,k}=0$ for $k \leq -1.$


Letting $k=1$ in  \eqref{1.303}, we obtain $(j-2)\beta_{j,1} = -j\beta_{j,0}.$
Finally, taking $k=i+1$ in equation  \eqref{1.2}, we derive
\begin{center}
$(i+1-2j)\beta_{i,i+1-j} +(i-1)\beta_{j,1}\ =(i-j)\beta'_{i+j,i+1}.
$\end{center} 
Since for the sufficiently large $i,$ we have $\beta_{i,i+1-j}=\beta'_{i+j,i+1}=0,$ it follows that $\beta_{j,1}=0$ for any $j.$  
Therefore, we obtain that $\beta_{j,k}=0$ for all  $j, k \in \mathbb{Z}.$

\end{enumerate}

\item Let $b\in \big\{0;1;2\big\}$ and $a \notin \mathbb{Z}.$ We replace $f$ by 
$f-{\rm ad}_x-\sum\limits_{k\in\mathbb{Z}}\eta_k g_k -\sum\limits_{k\in\mathbb{Z}}\theta_k {\mathfrak p}_{j,k}-  \beta_{0,0} h,$ where 
\begin{center}
$x=\sum\limits_{k\in\mathbb{Z}\setminus \{0\}} \frac{\beta_{0,k}}{k+a} I_k;$   \quad 
$\eta_0 = \beta_{1,1} - \beta_{0,0}, $\quad 
$\eta_k = \beta_{1,k+1} - \frac{k+a+b}{k+a} \beta_{0,k}, \ k \in \mathbb{Z}\setminus \{0\};$\quad
$\theta_k=\beta_{2,k+2}-2\beta_{1,k+1}+\beta_{0,k}, \ k \in \mathbb{Z}.$

\end{center}

Thus, we may assume $\beta_{0,k}=\beta_{1,k}=\beta_{2,k}=0$ for  $k\in\mathbb{Z}.$ 

Putting $i=0,$  and $i=1$ in \eqref{1.2}, we get 
\begin{center}
$\begin{array}{rcl}(k+a)\beta_{j,k} &= &j \beta'_{j,k},\\[1mm] (k+a+b-1)\beta_{j,k-1}&= &(j-1) \beta'_{j+1,k}.
\end{array}$ 
\end{center}

Thus, we obtain 
\begin{center}
$\begin{array}{rcl}(k+a)(j-2)\beta_{j,k} &= & (k+a+b-1) j \beta_{j-1,k-1}.
\end{array}$ 
\end{center}
Since $\beta_{2,k}=0$ and $a \notin \mathbb{Z},$ then recurrently we get $\beta_{j,k}=0$ for any $j\geq 3.$ Consequently, $\beta'_{j,k}=0$ for any $j\geq 3.$

Now, considering \eqref{1.2} for $j \leq -1, $ and $i$ such that $i+j \geq 1,$ we derive \begin{center}$(a+bi+k)\beta_{j,k}=0.$\end{center} Since, $a \notin \mathbb{Z},$ then  $a+bi+k\neq 0$  and we wet $\beta_{j,k}=0$ for any $j\leq -1.$  

Therefore, we obtain that $\beta_{j,k}=0$ for all  $j, k \in \mathbb{Z}.$

\item Let $b\notin \big\{ -1;0;1;2 \big\}$ and $a\notin \mathbb{Z}.$ We replace $f$ by 
$f-{\rm ad}_x-\sum\limits_{k\in\mathbb{Z}}\eta_k g_k - \beta_{0,0} h,$ where \begin{center}
$x=\sum\limits_{k\in\mathbb{Z}\setminus \{0\}}\frac{\beta_{0,k}}{k+a} I_k;$ \quad  
$\eta_0 = \beta_{1,1} - \beta_{0,0},$ \quad $\eta_k = \beta_{1,k+1} - \frac{k+a+b}{k+a} \beta_{0,k}, \ k \in \mathbb{Z}\setminus \{0\}.$
\end{center}

Thus, we may assume $\beta_{0,k}=\beta_{1,k}=0$ for  $k\in\mathbb{Z}.$

Putting $i=0$ and  $i=1$ in \eqref{1.2}, we get 
\begin{center}
$(k+a)\beta_{j,k}\ =\  j \beta'_{j,k}$ \ and \  
$(k+a+b-1)\beta_{j,k-1}\ =\  (j-1) \beta'_{j+1,k}.$ 
\end{center}

Since $a\notin\mathbb{Z},$ then we have
    $\beta_{j,k} \ = \ \frac{j(k+a+b-1)}{(j-2)(k+a)} \beta_{j-1,k-1}.$ Thus, by induction we get $$\beta_{j,k}\ = \  \frac{j(j-1)\prod\limits_{t=1}^{j-2}(k+a+b-t)}{2\prod\limits_{t=0}^{j-3}(k+a-t)} \beta_{2,k-j+2}, \quad j>2.$$ 

Putting $i=2$ and $j=-1$ in \eqref{1.2}, we obtain  
$(k+a+2b)\beta_{-1,k} \ =\ (k+a-b+3)\beta_{2,k+3},$ which implies $\beta_{2,k+3}=0$ for $k = -a-2b.$

Then putting $i=3$ and $j=-1$ in \eqref{1.2} for $k \neq -a-2b,$ we obtain
\begin{longtable}{lclcl}
$0$&$=$&$(k+a+3b) \beta_{-1,k}-(k+a-b+4)\beta_{3,k+4}+4\beta'_{2,k+3}$&$=$&$\frac{6b(b-1)(b-2)}{(k+a+4) (k+a+2b)}\beta_{2,k+3}.$
\end{longtable}
 
 Thus, we get $\beta_{2,k}=0$ for any $k \in \mathbb{Z},$
 i.e., $\beta_{j,k}=0$ for any $j \leq 2$ and $k \in \mathbb{Z}.$

 Now, considering \eqref{1.2} for $j \leq -1, $ and $i$ such that $i+j \geq 1,$ we derive \begin{center}
     $(a+bi+k)\beta_{j,k}=0.$
 \end{center} Since, $a \notin \mathbb{Z},$ then there exist $i$ such that $a+bi+k\neq 0.$ Hence, $\beta_{j,k}=0$ for any $j\leq -1.$  

Therefore, we obtain that $\beta_{j,k}=0$ for all  $j, k \in \mathbb{Z}.$

 \item  Let $b\in \big\{0;1;2\big\},$ and $a = 0.$ 
We replace $f$ by 
$f-{\rm ad}_x-\sum\limits_{k\in\mathbb{Z}}\eta_k g_k - \beta_{0,0} h,$ where \begin{center}
$x=\sum\limits_{k\in\mathbb{Z}\setminus \{0\}} \frac{\beta_{0,k}}{k} I_k;$ \quad  
$\eta_0 = \beta_{1,1} - \beta_{0,0},$ \quad 
$\eta_k = \beta_{1,k+1} - \frac{k+b}{k} \beta_{0,k}, \ k \in \mathbb{Z}\setminus \{0\}.$
\end{center}

Thus, we may assume $\beta_{0,k}=\beta_{1,k}=0$ for  $k\in\mathbb{Z}.$ 


Putting $i=0,$  and $i=1$ in \eqref{1.2}, we get 
\begin{center}
$\begin{array}{rcl}k \beta_{j,k} &= &j \beta'_{j,k},\\[1mm] (k+b-1)\beta_{j,k-1}&= &(j-1) \beta'_{j+1,k}.
\end{array}$ 
\end{center}

Thus, we obtain 
\begin{equation*}\begin{array}{rcl}k(j-2)\beta_{j,k} &= & (k+b-1) j \beta_{j-1,k-1}.
\end{array}  
\end{equation*}
\begin{enumerate} 
\item Let $b=2.$ Then we have
\begin{equation*}
\begin{array}{rcl} k \beta_{j,k} &= &j \beta'_{j,k},\\[1mm] (k+2)\beta_{j,k}&= &(j-1) \beta'_{j+1,k+1},
\end{array}\end{equation*}
and
\begin{equation}\label{2.b=2}\begin{array}{rcl}k(j-2)\beta_{j,k} &= & (k+1) j \beta_{j-1,k-1}.
\end{array} 
\end{equation}

Then again replacing $f$ by $f-\sum\limits_{k\in\mathbb{Z}\setminus \{-3\} }\beta_{2,k+2} {\mathfrak p}_{2,k},$  we may assume $\beta_{2,k}=0$ for  $k \neq -1.$
Moreover, taking $k=0$ and $j=3$, in \eqref{2.b=2}, we get $\beta_{2,-1} =0.$ 
Thus, we have  
$\beta_{0,k}=\beta_{1,k}=\beta_{2,k}=0$ for  all $k\in\mathbb{Z}.$

Then, from \eqref{2.b=2}, we obtain inductively that
\begin{longtable}{lcll}
$\beta_{j,k} $&$=$&$ \frac{j(j-1)(k+1)}{(j-k)(j-k-1)} \beta_{j-k,0}, $&$ 0 \leq k \leq  j-3;$\\ 
$\beta_{j,k} $&$=$&$0, $&$ j \geq 3 \  \text{and} \ k< 0 \ \text{or} \ k > j-3.$
\end{longtable}

Now, taking $i=-1,$ $j=2$ in equation \eqref{1.2}, we get $(k+4)\beta_{-1,k}=0,$ which implies $\beta_{-1,k}=0$ for $k \neq -4.$

Putting $i=-1,$ $j=3,$ $k=-1$ in equation \eqref{1.2}, we obtain $\beta_{-1,-4}= - \beta_{3,0}.$ Then using \eqref{2.b=2}, for $j\leq -1$ and $k=j-3,$ recurrently we obtain 
\begin{equation*}\label{1.244}
       \beta_{j,j-3} = \frac{j(j-1)(j-2)}{6} \beta_{3,0}, \ j \leq -1. 
    \end{equation*}

Now, considering equation \eqref{1.2} for $i\leq -2,$ $j=1-i,$ $k=-1-i,$ we obtain $$2 i \beta_{i,-2} - \beta_{1-i,-1-2i} - (1-2 i) \beta'_{1,-1-i}=0,$$ which implies $\beta_{i,-2}=0$ for $i\leq -2.$
Then using \eqref{2.b=2}, for $j\leq -1$ and $k\neq j-3,$ recurrently we obtain $\beta_{j,k}=0$ for $j\leq -1$ and $k\neq j-3.$

Finally, considering for $j \geq 4,$ $i=1-j,$ $k=1-j,$ in equation \eqref{1.2}, we obtain $$\beta_{1-j, 1-2 j} + 2 (j-1) \beta_{j, 0} + (2j-1 )\beta'_{1, 1 - j}=0,$$ which implies $\beta_{j, 0}=0$ for $j \geq 4.$ Hence, $\beta_{j,k}=0$ for $j\geq 4$ and $k\neq j-3.$

Therefore, in this case, we obtain the quasi-derivation $\mathfrak E.$ 

\item Let $b=1.$ Then we have
\begin{equation*}
\begin{array}{rcl} k \beta_{j,k} &= &j \beta'_{j,k},\\[1mm] (k+1)\beta_{j,k}&= &(j-1) \beta'_{j+1,k+1},
\end{array}\end{equation*}
and
\begin{equation}\label{2.b=1}\begin{array}{rcl}k(j-2)\beta_{j,k} &= & k j \beta_{j-1,k-1}.
\end{array} 
\end{equation}

Replacing $f$ by $f-\sum\limits_{k\in\mathbb{Z}}\beta_{2,k+2} {\mathfrak p}_{1,k},$   we may assume $\beta_{2,k}=0$ for  $k\in\mathbb{Z}.$
Thus, we have $\beta_{0,k}=\beta_{1,k}=\beta_{2,k}=0$ for  $k\in\mathbb{Z}.$

Putting $i=2$ and $j=-1$ in \eqref{1.2}, we get $\beta_{-1,k}=0$ for $k\neq -2.$

Putting $i=3,$ $j=k=-1$ in \eqref{1.2}, we get $\beta_{3,0}=0.$ 
Since $\beta_{2,k}=0,$ then recurrently we get $\beta_{j,k}=0$ for  $j\geq 3$  and $k<0$ or $k\geq j-3.$ 

Now, considering equation \eqref{1.2} for $i\leq -2,$ $j=1-i,$ $k=-1-i,$ we obtain $$(-1-i) \beta_{i,-2} +(1+i)\beta_{1-i,-1-2i} - (1-2 i) \beta'_{1,-1-i}=0,$$ which implies $\beta_{i,-2}=0$ for $i\leq -2.$

Then using \eqref{2.b=1}, for $j\leq -1$ and $k\neq0,$ recurrently we obtain $\beta_{j,k}=0$ for $j\leq -1,$ $k=j$ and $k>0$ or $k< j-1.$

Finally, considering for $j \geq 4,$ $i=1-j,$ $k=1-j,$ in equation \eqref{1.2}, we obtain $$(1-j)\beta_{1-j, 1-2 j} -(1-j) \beta_{j, 0} + (2j-1 )\beta'_{1, 1 - j}=0,$$ which implies $\beta_{j, 0}=0$ for $j \geq 4.$ 

Hence, by the  relation \eqref{2.b=1}, similarly to the previous case we deduce $\beta_{j,k}=0$ for all $j,k\in\mathbb{Z}.$




\item Let $b=0.$ Then, we have 

\begin{equation}\label{b=0}
\begin{array}{rcl}k \beta_{j,k} &= &j \beta'_{j,k},\\[1mm] k\beta_{j,k}&= &(j-1) \beta'_{j+1,k+1},
\end{array} 
\end{equation}
and 
\begin{equation}\label{case4_b=0}
k(j-2)\beta_{j,k} = (k-1) j \beta_{j-1,k-1}.\end{equation} 

Putting $k=j$ and $k=j-1$ in \eqref{case4_b=0}, recurrently, we obtain  
$$\beta_{j,j}=(j-1)\beta_{2,2}, \quad  \beta_{j,j-1}=\frac{j}{2}\beta_{2,1}, \quad j \geq 2.$$

Now, in equation \eqref{1.2}, substituting the triples $(i,j,k):=(1-j, j, 1)$ and $(i,j,k):=(1-j, j, 0),$ for $j \leq -1,$ we derive 
$$\beta_{j,j}=(j-1) \beta_{2,2}, \quad \beta_{j,j-1}=\frac{j}{2}\beta_{2,1}, \quad j \leq -1.$$

Where $$\beta'_{j,j} = (j-1)\beta_{2,2}, \quad \beta'_{j,j-1}=\frac{j-1}{2}\beta_{2,1}, \quad \beta'_{0,-1}=-\frac{1}{2}\beta_{2,1}.$$

Taking $k=0,1,-1$ in \eqref{b=0}, we derive 
\begin{equation}\label{2.b=0}
\begin{array}{lclllcll}
\beta_{j,-1} & =& 0 & \text{for} \  j\neq -1; &  \beta_{j,1} & =&0 &  \text{for} \ j\neq 2;\\
\beta'_{j,0}&=&0& \text{for} \  j\neq 0;&   
\beta'_{j,1}&=&0 &  \text{for} \ j\neq 2.
\end{array}
\end{equation}

Taking $i=-1,$ $k=0$ and $i=2,$ $k=1$ in \eqref{1.2}, respectively, we have 
\begin{longtable}{rcl}
$j\beta_{-1,-j}+\beta_{j,1}-(j+1) \beta'_{j-1,0} $&$=$&$0,$\\
$(1-j)\beta_{2,1-j}+\beta_{j,-1}+(j-2) \beta'_{2+j,1} $&$=$&$0,$
\end{longtable}
which implies 
\begin{equation}\label{3.b=0}\beta_{-1,j}=0\ \text{for} \  j\notin \big\{ 0; -1; -2\big\}, \qquad \beta_{2,j}=0\ \text{for} \  j\notin \big\{ 0; 1; 2\big\}.\end{equation}

Now, using \eqref{2.b=0} and \eqref{3.b=0} from \eqref{case4_b=0}, recurrently by $k$ for $k\leq -1$ and $k\geq 2$, we obtain that  
$$\beta_{j,k}=0\ \text{for} \  k\notin \big\{ 0; \ j; \ j-1\big\}.$$

Therefore, in this case, we have $f = {\mathfrak f}_1+{\mathfrak f}_2+{\mathfrak f}_3$ such that 
\begin{longtable}{|lcllcllcllcl|}
\hline
${\mathfrak f}_1(L_0) $&$=$&$0$ & ${\mathfrak f}_1(L_1) $&$=$&$0$ & ${\mathfrak f}_1(L_i) $&$=$&$  (i-1)I_i$ & ${\mathfrak f}_1(I_i)$&$ = $&$ 0$\\
\hline
${\mathfrak f}_1'(L_0) $&$=$&$-I_0$ & ${\mathfrak f}_1'(L_1) $&$=$&$0$ & ${\mathfrak f}_1'(L_i) $&$=$&$  (i-1)I_i$ & ${\mathfrak f}_1'(I_i)$&$ = $&$ 0$\\
\hline
${\mathfrak f}_2(L_0) $&$=$&$0$ & ${\mathfrak f}_2(L_1) $&$=$&$0,$ & ${\mathfrak f}_2(L_i) $&$=$&$  \frac{i} {2} I_{i-1}$ & ${\mathfrak f}_2(I_i)$&$ = $&$ 0$\\
\hline
${\mathfrak f}_2'(L_0) $&$=$&$-\frac{1} {2}I_{-1}$ & ${\mathfrak f}_2'(L_1) $&$=$&$0,$ & ${\mathfrak f}_2'(L_i) $&$=$&$  \frac{i-1} {2}I_{i-1}$ & ${\mathfrak f}_2'(I_i)$&$ = $&$ 0$\\
\hline
${\mathfrak f}_3(L_i) $&$=$&$  \beta_i I_0$ & ${\mathfrak f}_3(I_i)$&$ = $&$ 0$ & ${\mathfrak f}_3'(L_i) $&$=$&$  0$ & ${\mathfrak f}_3'(I_i)$&$ = $&$ 0$ \\
\hline
\end{longtable}
\noindent
It is easy to see that
\begin{longtable}{c}
    ${\mathfrak f}_1 \in \mathfrak{Der} \big({\mathcal W}(0,0)\big)+{\rm Ann}^Q\big({\mathcal W}(0,0)\big);$
${\mathfrak f}_2 \in \langle g_{-1}\rangle +{\rm Ann}^Q\big({\mathcal W}(0,0)\big);$
\mbox{ and }${\mathfrak f}_3 \in  {\rm Ann}^Q\big({\mathcal W}(0,0)\big).$
\end{longtable}

\end{enumerate}

\item Let $b\notin\big\{ -1;0;1;2\big\}$ and  $a = 0.$ In this case as well, similar to the situation where 
$a\notin\mathbb{Z},$ we reduce to the case of   $\beta_{0,j}=\beta_{1,j}=0$ for all  $j\in\mathbb{Z}.$
 Then we have
   \begin{center}
$\begin{array}{rcl}k \beta_{j,k} &= &j \beta'_{j,k},\\[1mm] (k+b)\beta_{j,k}&= &(j-1) \beta'_{j+1,k+1},
\end{array}$ 
\end{center}
 and   
\begin{equation} \label{case5} k(j-2)\beta_{j,k} = j(k+b-1) \beta_{j-1,k-1}.
\end{equation}
For $k=0$ we have $j(b-1)\beta_{j-1,-1}=0,$ it gives 
$\beta_{j,-1}=0,$ $j\neq -1.$

Putting $i=2,$  $j=-1$  and $i=3,$ $j=-1$ in \eqref{1.2}, we get the following relations, respectively 
\begin{equation}\label{1c}
 (k-b)\beta_{2,k}-(k-3+2b)\beta_{-1,k-3}=0,   
\end{equation}
\begin{equation}\label{2c}
(k-b+1)\beta_{3,k+1}-(k-3+3b)\beta_{-1,k-3}=4\beta'_{2,k}.    
\end{equation}
Combining    \eqref{1c} and \eqref{2c} together with $\beta_{3, k + 1} = \frac{3(b + k)}{k + 1} \beta_{2, k},$ $k\neq -1$ and using $b\notin\{0;1;2\}$   we obtain $(-4 + 2 b) \beta_{-1, -3 + k} = 0$ for $k\neq -1.$ Taking $k=-1$ in \eqref{1c}, we get $\beta_{-1, -4}=0,$  from this it follows that $\beta_{-1, k}=0,$ also  $\beta_{2,k}=0$ for $k\in\mathbb{Z}.$ 

Then setting $k = -1$ in \eqref{2c}, we have $\beta_{3,0}=0.$ Hence for determining the coefficients $\beta_{3,k},$ $k\neq 0,$ we set $j=3$ in \eqref{case5}, it implies $k \beta_{3,k}= 3(b + k - 1)  \beta_{2,k-1} = 0.$  So we deduce $ \beta_{3,k}=0$ for $k\in\mathbb{Z}.$
   
 Now, in the equation \eqref{case5}, substituting $(j,k):=(4,k+4),$    we derive 
\begin{center}    $(k + 3) \beta_{4, k + 3} = 2 (b + k + 2) \beta_{3, k + 2},$
\end{center} it implies $\beta_{4,k+3}=0$ for $k\neq -3.$ In the equation \eqref{1.2}, substituting the triple $(i,j,k):=(4,-1,-1),$    we derive 
    \begin{center}
        $0\ =\ (5 - 4 b) \beta_{-1, -5} - b \beta_{4, 0} - 
   5 \beta'_{3, -1}\ = \ -b \beta_{4, 0}\ =\ 0.$
    \end{center} It follows that $\beta_{4,k}=0$ for all $k\in\mathbb{Z}.$
 
Putting $i=3,$ $j=-2$ and $i=4,$ $j=-2$ in \eqref{1.2}, we have 
\begin{longtable}{lclcl}
$0$&$=$&$ (3 - 3 b - k) \beta_{-2, -3 + k} + (2 - 2 b + 
      k) \beta_{3, 2 + k} - 5 \beta'_{1, k} $&$=$&$ (3 - 3 b - k) \beta_{-2, -3 + k};$\\  
$0$&$ =$&$ (3 - 4 b - k) \beta_{-2, -3 + k} + (3 - 2 b + k) \beta_{
     4, 3 + k} - 
   6 \beta'_{2, k + 1} $&$=$&$  (3 - 4 b - k) \beta_{-2, -3 + k} .$
   \end{longtable}
Combining these relations we obtain $\beta_{-2, -3 + k}=0 $ for  $k\in\mathbb{Z}.$ Hence 
$\beta_{-2,  k}=0 $ for  $k\in\mathbb{Z}.$ 

Now, by induction on $j,$ we show that all coefficients $\beta_{j,k}$ are 0.
 Thus, let $\beta_{j-1,k}=0,$ $j\geq2,$ $k\in\mathbb{Z},$ then from \eqref{case5}   we get 
 $k (j - 2) \beta_{j, k} = 0,$ it implies $ \beta_{j, k} = 0,$ $ k \neq 0. $
Putting $i=k=-1,$  in \eqref{1.2}, we obtain 
$$0 \ = \ \big(-1 + (-1 + b) j \big) \beta_{-1, -1 - j} + 
    b \beta_{j, 0} + (1 + j) \beta'_{j-1, -1}\ =\  b \beta_{j, 0},$$ so $\beta_{j, 0}=0,$ it follows that  $ \beta_{j, k} = 0,$ $j\geq2$ and  $k\in\mathbb{Z}.$ 

 Next, we consider the coefficients $\beta_{-j,k} $ for  $j\geq2,$ $k\in\mathbb{Z}.$ Let  $\beta_{2-j,k}=0,$ $j\geq2,$ $k\in\mathbb{Z},$
 then in the equation \eqref{1.2}, we substitute the triple with $(i,j,k):=( 1 - j,j,k)$ and   $(i,j,k):=( 1 - j,j+1,k),$  we obtain 
 \begin{longtable}{lclcl}$0$&$ =$&$ ((b-1) j + k) \beta_{1 - j, k-j } - (b-1 + j - b j + k) \beta_{
     j, -1 + j + k} + (2 j-1) \beta'_{1, k} $\\ &$=$&$ ((b-1) j + k) \beta_{1 - j, k-j};$\\
$ 0 $&$=$&$ (b + k - j + b j) \beta_{1 - j, k - j} - (b + k + j - b j) \beta_{1 + j, 
     k + j}+ 2 j \beta'_{2, k + 1} $\\ &$=$&$ (b + k - j + b j) \beta_{1 - j, 
     k - j}.$
     \end{longtable}
Combining these relations we obtain $\beta_{1-j,  k-j}=0 $ for  $k\in\mathbb{Z}.$ Hence, by induction on $j$ we deduce that all 
$\beta_{j,  k}=0 $ for  $j,k\in\mathbb{Z}.$ 
\end{enumerate}
\end{proof}

\begin{theorem}\label{MT}

  $ {\rm Q}\mathfrak{Der} \big({\mathcal W}(a,b)\big) \ =\ 
  \mathfrak{Der}  \oplus 
  \mathfrak{Der}_{\frac 12}   \oplus 
  \mathfrak{Der}_{\frac 1{1-b}}  
  \oplus  \big\langle {\mathfrak F}_j \big\rangle
  \oplus  \big\langle {\mathfrak p}_{b,j} \big\rangle  
  \oplus {\rm Ann}^Q,$ where 
${\rm Ann}^Q \neq 0$   only if $(a,b)=(0,0),$
$ \mathfrak{Der}_{\frac 1{1-b}}   = 
\big\langle {\mathfrak G}_j \ |\  {\mathfrak G}_j(L_i)= I_{i+j},\ {\mathfrak G}_j(I_i) = 0\big\rangle$\footnote{ 
$\mathfrak{Der}_{\infty} :=\{ f \in {\rm End} \ |\  [f(x),y]+[x,f(y)]=0\}.$ } and

\begin{longtable}{|rclrcllll|}
\hline
${\mathfrak F}_j(L_i) $&$=$&$  i L_{i+j},$ & ${\mathfrak F}_j(I_i)$&$ = $& \multicolumn{4}{l|}{$ (a+i) I_{i+j}$ }\\
${\mathfrak F}'_j(L_i) $&$=$&$  (i+j) L_{i+j},$& ${\mathfrak F}'_j(I_i) $&$=$&\multicolumn{4}{l|}{$  (a+i+j) I_{i+j}$ }\\
\hline

${\mathfrak p}_{0,j}(L_i)$ & $=$ & $\frac{i(i-1)}{(a+i+j)}I_{i+j}$, & ${\mathfrak p}_{0,j}(I_i)$ & $=$ & $0$ & \text{for}&  $a\notin \mathbb{Z},$  &$b=0$\\
${\mathfrak p}'_{0,j}(L_i)$ & $=$ & ${(i-1) }I_{i+j}$, & ${\mathfrak p}'_{0,j}(I_i)$ & $=$ & $0$ & \text{for}& $a\notin \mathbb{Z}$ & $b=0$ \\ \hline

${\mathfrak p}_{1,j}(L_i)$ & $=$ & $i^2 I_{i+j}$, & ${\mathfrak p}_{1,j}(I_i)$ & $=$ & $0$ & \text{for}& & $b=1$\\
${\mathfrak p}'_{1,j}(L_i)$ & $=$ & ${i(a+i+j)}I_{i+j}$, & ${\mathfrak p}'_{1,j}(I_i)$ & $=$ & $0$ & \text{for}& & $b=1$\\
\hline
 
${\mathfrak p}_{2,j}(L_i)$ & $=$ & ${i(i-1)(a+1+i+j)}I_{i+j}$, & ${\mathfrak p}_{2,j}(I_i)$ & $=$ & $0$ & \text{for}&      & $b=2$\\
${\mathfrak p}'_{2,j}(L_i)$ & $=$ & ${(i-1)(a+i+j)(a+1+i+j)}I_{i+j}$, & ${\mathfrak p}'_{2,j}(I_i)$ & $=$ & $0$ & \text{for}&   & $b=2$ \\
\hline

\end{longtable}

In particular, 
$ {\rm Q}\mathfrak{Der} \big({\mathcal W}(a,-1)\big) \ =\ 
   \mathfrak{Der} \big({\mathcal W}(a,-1)\big)   \oplus 
  \mathfrak{Der}_{\frac 12} \big({\mathcal W}(a,-1)\big).$


\end{theorem}
\begin{proof}

 First of all, we rewrite  the mapping from proposition \ref{lemma11} in a more compact way:
 mappings ${\mathfrak p}_{i,j}$ will be multiplied by a suitable constant and joined with $\mathfrak E;$
 mappings $g_j$ and $h$ give (modulo a suitable derivation) mappings $\mathfrak G_j.$
 
 By   Lemma \ref{lemma}, we may assume that $a=0,$ if $a\in\mathbb{Z}.$
   Let $f \in  {\rm Q}\mathfrak{Der}\big({\mathcal W}(a,b)\big),$ then 
\begin{longtable}{rclrcl}
$f(L_i) $&$ = $&$ \sum\limits_{k\in\mathbb{Z}}\alpha_{i,k}L_k+\sum\limits_{k\in\mathbb{Z}}\beta_{i,k}I_k, $&$ 
f(I_i) $&$ = $&$  \sum\limits_{k\in\mathbb{Z}}\gamma_{i,k}I_k+\sum\limits_{k\in\mathbb{Z}}\sigma_{i,k}L_k,$\\
$f'(L_i) $&$ = $&$  \sum\limits_{k\in\mathbb{Z}}\alpha'_{i,k}L_k+\beta'_{i,k}I_k, $&$ f'(I_i) $&$  =$&$  \sum\limits_{k\in\mathbb{Z}}\gamma'_{i,k}I_k+\sigma'_{i,k}L_k.$
\end{longtable}

\medskip 
\noindent
{\bf Step $\alpha$:} Reduction to the case of $\alpha_{i,j}=0.$

Given that ${\mathcal W}(a,b)$ is a $\mathbb{Z}_2$-graded algebra, it is possible to express $f$
as the sum of two quasi-derivations $f_0$  and $f_1,$ such that 
\begin{longtable}{rclrcl}
$f_0\big(\langle{L_i}\rangle_{i\in\mathbb{Z}}\big)$&$\subseteq$&$\langle{L_i}\rangle_{i\in\mathbb{Z}},$&
$f_0\big(\langle{I_i}\rangle_{i\in\mathbb{Z}}\big) $&$\subseteq$&$\langle{I_i}\rangle_{i\in\mathbb{Z}},$\\
$f_1\big(\langle{L_i}\rangle_{i\in\mathbb{Z}}\big)$&$\subseteq$&$\langle{I_i}\rangle_{i\in\mathbb{Z}},$&$ f_1\big(\langle{I_i}\rangle_{i\in\mathbb{Z}}\big)$&$\subseteq$&$\langle{L_i}\rangle_{i\in\mathbb{Z}}.$
\end{longtable}
We begin by finding the form of $f_0.$ 
Let $f_{\rm W} = f_0{\big|}_{\langle{L_i}\rangle_{i\in\mathbb{Z}}}$
and $f_{\rm W}$ is a quasi-derivation of the Witt algebra ${\rm W}.$
Theorem \ref{qWitt} gives the description of $f_{\rm W}$ as a sum of derivations and $\frac{1}{2}$-derivations.
It is easy to see that there is   $f^*_{\rm W} \in  {\rm Q}\mathfrak{Der}\big({\mathcal W}(a,b)\big),$ constructed as a sum of mappings from Proposition \ref{lemma11}, such that 
 $f_{\rm W}^*{\big|}_{\langle{L_i}\rangle_{i\in\mathbb{Z}}} = f_0{\big|}_{\langle{L_i}\rangle_{i\in\mathbb{Z}}}.$
We can replace $f_0$ by $f_0-f^*_{\rm W}$. 
Hence,  we can assume that $\alpha_{i,k}=0$ for all $i, k\in\mathbb{Z}.$

\medskip 

According to the definition of  
    quasi-derivations, the coefficients  
must satisfy the following system of equations:
\begin{longtable}{rcl}
$\sum\limits_{ k\in\mathbb{Z}}\left((k+a+bj)\beta_{i,k}I_{j+k} -(k+a+bi)\beta_{j,k} I_{i+k}\right)$&$=$&$(i-j)\sum\limits_{ k\in\mathbb{Z}}\beta'_{i+j,k}I_k,$ \\
$\sum\limits_{ k\in\mathbb{Z}}\left((i+a+bk)\sigma_{j,k}I_{i+k} -(j+a+bk)\sigma_{i,k} I_{j+k}\right)$&$=$&$0,$ \\
$\sum\limits_{ k\in\mathbb{Z}}(i-k)\sigma_{j,k}L_{i+k}$&$=$& $-(j+a+bi)\sum\limits_{ k\in\mathbb{Z}}\sigma'_{i+j,k} L_{k},$\\ 
$\sum\limits_{ k\in\mathbb{Z}}(k+a+bi)\gamma_{j,k} I_{i+k}$&$=$&$(j+a+bi)\sum\limits_{ k\in\mathbb{Z}}\gamma'_{i+j,k}I_k.$ 
\end{longtable}
We obtain  the following equations by comparing the coefficients of the basis:
\begin{eqnarray}
\label{1.3} (i+a+b(k-i))\sigma_{j,k-i} -(j+a+b(k-j))\sigma_{i,k-j}\ =&0, \\[1mm]
\label{1.4} (2i-k)\sigma_{j,k-i}\ =& -(j+a+bi)\sigma'_{i+j,k} ,\\[1mm] 
\label{1.5}  (k-i+a+bi)\gamma_{j,k-i}\ =&(j+a+bi)\gamma'_{i+j,k}. 
\end{eqnarray}

\noindent
{\bf Step $\beta$:} We find   $\beta_{i,j}.$
It follows from Lemma \ref{lemmabeta}.

\medskip 
\noindent
{\bf Step $\gamma$:} Reduction to the case of  $\gamma_{i,j}=0.$

\noindent  
Replacing $f$ by $f-\gamma_{0,0}D_1,$ we can suppose that $\gamma_{0,0}=0.$

\begin{enumerate}
    \item[$(\gamma_1)$] If $a=0$ and $b\notin \{0,1\},$ then  
    letting $i=j=0$
in \eqref{1.5}, we obtain $\gamma_{0,k}=0$ for $k\neq 0,$
 i.e., $\gamma_{0,k}=0$ for all $k\in \mathbb Z.$
Putting $j=0$ in \eqref{1.5},  it gives 
\begin{equation}\label{1.9''}
  i\gamma'_{i,k}=0, \quad \mbox{ for }i\neq 0.
\end{equation} 
Putting $i=0$ in \eqref{1.5}, we have 
\begin{equation}\label{1.10''}
k\gamma_{j,k}=j\gamma'_{j,k}=0, \quad\mbox{ for }  0\notin \{j,k\}.
\end{equation}
Taking $j=-i$ in \eqref{1.5}, we get $\gamma'_{0,k}=0,$  for all $k\in\mathbb{Z}.$ 
Taking $k=i$ in \eqref{1.5}, we get  $\gamma_{j,0}=0,$  for $j\neq 0.$ 
Putting $k=i$ in \eqref{1.9''} and $k=j$ in \eqref{1.10''}, we have  
$\gamma'_{i,i}=\gamma_{i,i}=0,$  for $i\neq 0.$
Hence, $\gamma_{i,j}=0$ for all $i,j \in \mathbb Z.$ 

  \item[$(\gamma_2)$] If $(a,b)=(0,0),$   then  
$(k-i) \gamma_{j,k-i}=j \gamma'_{i+j,k}.$
Taking $k=i,$ we have that $\gamma'_{i+j,i}=0$ for $j\neq0.$ 
Hence $\gamma_{i,j}=0$ of $i\neq j$ and $\gamma_{j,j}=\gamma_{i+j,i+j}',$ that gives $\gamma_{i,i}=\gamma_{j,j},$ for all $i,j \in \mathbb Z;$  
i.e. $\gamma_{i,i}=0$ for all $i\in \mathbb Z.$  Hence, $\gamma_{i,j}=0$ for all $i,j \in \mathbb Z.$   

 \item[$(\gamma_3)$] If $(a,b)=(0,1),$   then  
$k \gamma_{j,k-i}=(i+j) \gamma'_{i+j,k}.$
Taking $j=-i,$ we have that $\gamma_{j,k+j}=0$ for $k\neq0.$ 
Hence $\gamma'_{i,j}=0$ of $i\neq 0$ and $\gamma_{i,i}=\gamma_{j,j},$ for all $i,j \in \mathbb Z;$  
i.e., $\gamma_{i,i}=0$ for all $i\in \mathbb Z.$  Hence, $\gamma_{i,j}=0$ for all $i,j \in \mathbb Z.$   
\footnote{Let us note that the present case gives an exceptional quasi-derivation $f=0$ with related $f'\neq0.$
Due to non-perfectless of ${\mathcal W}(0, 1),$ the mapping $f'$ can be arbitrary satisfying 
$f'[{\mathcal W}(0, 1),{\mathcal W}(0, 1)]=0,$ but  $f'({\mathcal W}(0, 1))\neq0.$}

\item[$(\gamma_4)$] 
If $a \notin {\mathbb Z},$ then  putting $i=0$ in \eqref{1.5}, we obtain $(k+a)\gamma_{j,k}=(j+a)\gamma'_{j,k}.$ 
 Letting $j=0$ in \eqref{1.5}, we have 
 $$(k-i+a+bi)\gamma_{0,k-i} =(a+bi)\gamma'_{i,k},$$ it implies  
$$(i+a)(k-i+a+bi)\gamma_{0,k-i} =(a+bi)(k+a)\gamma_{i,k}.$$ 
From this equation we get
\begin{equation}\label{103}
\gamma_{i,k}\ =\ \frac{(i+a)(k-i+a+bi)}{(a+bi)(k+a)}\gamma_{0,k-i}, \quad  \mbox{ if }bi\neq- a.
\end{equation} 
Taking $i=-j$ in \eqref{1.5}, we get 
$$(k+j+a-bj)\gamma_{j,k+j} \ =\ (j+a-bj)\gamma'_{0,k},$$
then $$(k+j+a-bj)\gamma_{j,k+j}\ =\ \frac{(j+a-bj)(k+a)}{a}\gamma_{0,k},$$
for $k\neq 0$ and  $a\neq 0,$ it implies
$$\gamma_{j,k+j}\ =\ \frac{(j+a-bj)(k+a)}{a(k+j+a-bj)}\gamma_{0,k}, \quad \mbox{ if } k\neq bj-a-j.$$
Substituting $k$ with $k-j$ we get 
\begin{equation}\label{104}
\gamma_{j,k}\ =\ \frac{(j+a-bj)(k-j+a)}{a(k+a-bj)}\gamma_{0,k-j}, \mbox{ if }
k\neq bj-a.     \end{equation} By the equalities  \eqref{103} and \eqref{104}, we deduce $\gamma_{0,k-j}=0$ for $k\neq j,$   $bj\neq-a,$ $k\neq bj-a.$ 
Thus, $$\gamma_{j,k}=\gamma'_{j,k}=0, \quad k\neq j, \quad bj\neq-{a},\quad k\neq b j-a. $$
If $bj=-{a},$ then $b\neq 0,$ and by choosing suitable  $i\neq 0,$ 
\eqref{1.5}  implies 
$\gamma_{-\frac{a}{b},k}=0.$ 
Similarly, we get $\gamma_{j,k}=0$ for $k= b j-a.$

At the same time,  \eqref{103} and \eqref{104} give  $\gamma_{i,i}=\gamma_{0,0}$ for all $i \in \mathbb Z,$ i.e., $\gamma_{i,i}=0$ for all $i\in \mathbb Z.$  
Hence, $\gamma_{i,j}=0$ for all $i,j \in \mathbb Z.$   
 \end{enumerate}

\medskip 

\noindent
{\bf Step $\sigma$:} We prove that $\sigma_{i,j}=0.$

\begin{enumerate}
    \item[$(\sigma_1)$] If $(a,b)\neq (0,0)$ and $b\neq 1.$
Taking $j=0$ in  \eqref{1.3} , it gives 
\begin{equation}\label{1.6}
   \big(i+a+b(k-i) \big)\sigma_{0,k-i} -(a+bk)\sigma_{i,k}\ =\ 0.
\end{equation}
From \eqref{1.6} we have  
\begin{equation}\label{1.7}
  \sigma_{i,k}=\frac{ a+bk-i(b-1)}{a+bk}\sigma_{0,k-i},\mbox{\ \ for \ }a\neq -bk.
\end{equation}
Using   \eqref{1.7}  in \eqref{1.3},  we obtain 
\begin{equation}\label{1.8}
  \frac{ bij(i-j)(b-1)}{(a+b(k-i))(a+b(k-j))}\sigma_{0,k-i-j}=0,
  \mbox{\ \ for \ }a\neq b(i-k), a\neq b(j-k).
\end{equation}
By \eqref{1.8} we obtain $\sigma_{0,k-i-j}=0,$  it gives $ \sigma_{i,k}=0.$ Then by \eqref{1.4} it implies 
$ \sigma'_{i,k}=0.$ 

 \item[$(\sigma_2)$]  If $(a,b)=(0,0),$ then substituting $j$ with $i$ in \eqref{1.4} it gives 
\begin{equation}\label{26ai}
(2j-k)\sigma_{i,k-j}=-i\sigma'_{i+j,k},  
\end{equation}
from \eqref{1.4}  and \eqref{26ai} we get  
\begin{equation}\label{27ai}
j(2j-k)\sigma_{i,k-j}=i(2i-k)\sigma_{j,k-i},  
\end{equation} for all $k\in\mathbb{Z},$ $i,j\neq0.$
Using \eqref{1.3}  in \eqref{27ai} we obtain  
$j(i-j)\sigma_{i,k-j}=0,$ then $\sigma_{i,k}=0$ for $i\neq 0.$
Putting $i=0$ in  \eqref{1.3} , it gives $\sigma_{0,k-j}=0,$ 
it implies $\sigma'_{i,k}=0$ for all $k\in\mathbb{Z}.$

 \item[$(\sigma_3)$]  If $b=1,$ then by \eqref{1.3}, we have $(k+a)(\sigma_{j,k-i}-\sigma_{i,k-j})=0,$ if $(k,a)\neq (0,0).$ 
 Thus $\sigma_{j,k-i}=\sigma_{i,k-j},$ for $(k,a)\neq(0,0).$ 
 It implies $\sigma_{i,k}=\sigma_{0,k-i}$ for all $k\in\mathbb{Z}$ and  $(k,a)\neq(0,0).$ 
 If $(k,a)=(0,0),$ then $\sigma_{j,i-i}=\sigma_{i,i-j}$ gives $\sigma_{j,0}=\sigma_{0,-j}.$ 
Setting $j=0$ in \eqref{1.4}  it gives 
\begin{equation}\label{31ai}
 (2i-k)\sigma_{0,k-i}=-(a+i)\sigma'_{i,k}.   
\end{equation}
Setting $i=0$, it gives 
\begin{equation}\label{32ai}
 -k\sigma_{j,k}=-(a+j)\sigma'_{j,k}.   
\end{equation}
From \eqref{31ai} and \eqref{32ai}, we obtain  
 $\sigma_{0,k-i}=0,$ hence we deduce 
$\sigma_{i,k}=0 $ and $\sigma'_{i,k}=0$ for all $i,k\in\mathbb{Z}.$ 
\end{enumerate}
\medskip

\end{proof}

\begin{corollary}
 ${\mathcal W}(a,b)$ admits a nontrivial
    transposed $\delta$-Poisson structure $\cdot$:
\begin{enumerate}
    \item if $b\notin \big\{ -1;1\big\},$ then  $\delta=\frac 1{1-b}$ and 
    there is a finite set $\{ \mu_k\}_{k \in \mathbb Z},$ such that
\begin{center}
    $L_i \cdot L_j \ =\  \sum\limits_{k \in \mathbb Z} \mu_k I_{i+j+k}.$    
\end{center}

    \item if $b= 1,$ then   $\delta=1$ and 
    there is a finite set $\{ \mu_k\}_{k \in \mathbb Z},$ such that
\begin{center}
    $L_i \cdot L_j \ =\  \sum\limits_{k \in \mathbb Z} (i+j+k+\mu_k) 
    I_{i+j+k}.$   

\end{center}
   \item if $b = -1,$ then  $\delta=\frac 12 $ and 
    there is a finite set $\{ \nu_k,\mu_k\}_{k \in \mathbb Z},$ such that
\begin{longtable}{rcl}
$L_i \cdot L_j$&$ =$&$ \sum\limits_{k \in \mathbb Z} \nu_k L_{i+j+k} + \sum\limits_{k \in \mathbb Z} \mu_k I_{i+j+k};$\\    
$L_i \cdot I_j $&$=$&$ \sum\limits_{k \in \mathbb Z} \nu_k I_{i+j+k}.$    
\end{longtable}

\end{enumerate}

\end{corollary}

\begin{proof}
    The description of all biderivations of   ${\mathcal W}(a,b)$ is given in \cite{tang2}.
    The description of transposed $\frac 1{1-b}$-structures is based on the classification of $\delta$-derivations of  ${\mathcal W}(a,b)$ (see, Theorem \ref{MT}) and similar to the case of $b=-1$ considered in \cite[Theorem 25]{FKL}.
\end{proof}

\end{document}